\documentclass[a4paper,11pt]{amsart}
\usepackage[T1]{fontenc}
\usepackage[english]{babel}
\usepackage[cp1252]{inputenc}
\usepackage{amsthm}
\usepackage{amsmath}
\usepackage{amsfonts}

\usepackage{xcolor,graphicx,float,hyperref,booktabs,array,listings}
\usepackage{tikz}
\hypersetup{colorlinks=true,linkcolor=blue,citecolor=blue,urlcolor=blue}
\definecolor{skyblue}{RGB}{113,193,225}
\definecolor{darkred}{RGB}{125,0,30}
\lstdefinestyle{exactpython}{language=Python,basicstyle=\ttfamily\scriptsize,
breaklines=true,breakatwhitespace=false,columns=fullflexible,showstringspaces=false,
frame=single,numbers=left,numberstyle=\tiny,xleftmargin=1.6em,framexleftmargin=1.2em}

\usepackage{tikz-cd}
\usepackage{pgfplots}
 \usepackage{microtype}
\usepackage[margin=0.8in]{geometry}

\usepackage{xcolor}
 \pgfplotsset{compat=1.18}
\usepgfplotslibrary{fillbetween}

\usepackage{amssymb}
\usepackage{hyperref}
\usepackage{color}
\usepackage{caption}
\usepackage{indentfirst}
\usepackage{amssymb}
\usepackage{eufrak}
\usepackage{mathrsfs}
\usepackage{xypic}

 \usepackage{booktabs,array,graphicx,longtable}
\usepackage{array,booktabs,tabularx}

\usepackage{booktabs}
\usepackage{listings}
\theoremstyle{plain}   

\usepackage{booktabs}
\usepackage{graphicx}
\usepackage{array}

\definecolor{amoebadarkblue}{RGB}{0,63,127}
\definecolor{contourred}{RGB}{224,0,0}

\newif\ifhighdensityplot
\highdensityplotfalse

\definecolor{ContourColor}{RGB}{100,0,10}

\newcommand{\R}{\mathbb R}
\newcommand{\C}{\mathbb C}

\newcommand{\Z}{\mathbb Z}

\newcommand{\Log}{\mathrm{Log}\,}
\newcommand{\A}{\mathscr A}

\newcommand{\Cont}{\mathcal C}

\newcommand{\Crit}{\operatorname{Crit}}

\newcommand{\Area}{\operatorname{Area}}

\newcommand{\Newt}{\operatorname{Newt}}

\usepackage{array}
\usepackage{tabularx}

 \usepackage{caption}
 
\newcommand\restr[2]{{
  \left.\kern-\nulldelimiterspace 
  #1 
  \right|_{#2} 
  }}

\newtheorem{remark}{Remark}[section]
\newtheorem*{mtheorem*}{Main Theorem}
\newtheorem{theorem}{Theorem}[section]

\newtheorem{proposition}{Proposition}[section]

\newtheorem{lemma}{Lemma}[section]

\begin{document}
\title{Beyond Harnack Amoeba Boundaries}
\author{Mounir Nisse}
\date{}

\address{Mounir Nisse\\
Department of Mathematics, Xiamen University Malaysia, Jalan Sunsuria, Bandar Sunsuria, 43900, Sepang, Selangor, Malaysia.
}

 \email{\href{mailto:mounir.nisse@gmail.com}{mounir.nisse@gmail.com}, {mounir.nisse@xmu.edu.my}}

 \urladdr{www.math.tamu.edu/\~{}nisse}
 
\thanks{This research is supported in part by Xiamen University Malaysia Research Fund (Grant no. XMUMRF/ 2024-C5/IMAT/0013).}

 

\maketitle
 
\begin{abstract}
We construct smooth irreducible plane curves whose complete amoeba contour is a smooth embedded curve equal to the amoeba boundary, although the curves are not torically equivalent to simple Harnack curves.  Finite monomial coverings yield explicit infinite families with higher logarithmic multiplicity and nonmaximal normalized amoeba area.  We classify the Newton polygons admitting such covering constructions through their exponent lattices, give an exact connectedness criterion in terms of monodromy, and treat primitive polygons by a separate phase-sensitive construction.  Exact triangular and trapezoidal families are complemented by positive-genus investigations exhibiting bounded amoeba-complement components.  These results show that a perfectly exposed amoeba contour does not determine the real structure, logarithmic sheet number, or Harnack type of the underlying curve.
\end{abstract}

\section{Introduction}

Amoebas compress complex algebraic geometry into real geometry without erasing the combinatorics at infinity.  If $C=V(f)\subset(\C^*)^2$ is a plane algebraic curve, its amoeba is the closed subset $\A_f=\Log(C)\subset\R^2$, where $\Log(z,w)=(\log|z|,\log|w|)$.  The Newton polygon $\Newt(f)$ controls the asymptotic directions of the amoeba, while the convex components of $\R^2\setminus\A_f$ are detected by the order map and indexed by lattice points of $\Newt(f)$ \cite{FPT,GKZ,MikhalkinSurvey}.  This interaction between complex geometry, convexity, and lattice combinatorics has made amoebas a natural meeting point of toric geometry, real algebraic geometry, and tropical geometry.

The boundary of an amoeba carries especially delicate information.  Let $S_f=\operatorname{Crit}(\Log|_C)$ be the logarithmic critical locus and let $\Cont(\A_f)=\Log(S_f)$ be the contour.  At a smooth point of $C$, logarithmic criticality is detected by the logarithmic Gauss map $\gamma_f:C\to\mathbb P^1(\C)$, defined by $\gamma_f(z,w)=[zf_z:wf_w]$: one has $S_f=\gamma_f^{-1}(\mathbb P^1(\R))$.  The topological boundary $\partial\A_f$ is contained in the contour, but the containment is often strict.  Critical branches may lie inside the amoeba, distinct critical points may have the same logarithmic image, and the contour may develop crossings, tangencies, cusps, or higher-order contacts even when the algebraic curve itself is smooth.  Lang, Shapiro, and Shustin placed this critical-value geometry at the center of their study of the real degree of amoeba contours and emphasized the need for effective hypotheses guaranteeing regularity \cite{LSS2021}.  The analytic distinction between the smoothness of $S_f$ and the smoothness of its image is developed further in \cite{NisseAnalyticity}.

Simple Harnack curves provide the classical model in which these structures align.  Their real loci occupy an extremal position in the associated toric surface, their logarithmic map is at most two-to-one, and their complete contour coincides with the amoeba boundary \cite{Mikhalkin2000}.  The same extremality is visible through area: for a plane curve with Newton polygon $\Delta$, the amoeba satisfies $\operatorname{Area}(\A_f)\leq\pi^2\operatorname{Area}(\Delta)$, and equality is characterized by Harnack geometry up to the appropriate real normalization \cite{MR2001}.  The Ronkin function and its Monge--Amp\`ere measure connect this area theory to subdivisions and triangulations of the Newton polygon \cite{PassareRullgard}.  These results make it natural to ask whether the apparently rigid identity $\Cont(\A_f)=\partial\A_f$, strengthened by smoothness of the contour, already forces the curve to be simple Harnack.

The purpose of this paper is to show that it does not.  The contour records the location of logarithmic critical values, but it does not retain the full organization of the logarithmic sheets above the amoeba.  It also forgets phase data that distinguish a real curve from a curve with no compatible toric real form.  Consequently, two curves may have equally regular boundary-only contours while having different generic fiber multiplicities, different normalized amoeba areas, and different toric-equivalence classes.  The examples constructed here make this loss of information explicit and measurable.

The first mechanism is transport by a finite monomial covering.  For an integral matrix $A\in\operatorname{Mat}_{2\times2}(\Z)$ with $\det A\neq0$, the map $\pi_A:(\C^*)^2\to(\C^*)^2$ defined by monomials is an unramified covering of degree $|\det A|$.  Its decisive feature is the exact identity $\Log\circ\pi_A=A\circ\Log$.  Thus inverse images of amoebas, their boundaries, logarithmic critical loci, and contours are transported by the same invertible real linear map.  When the inverse image of a smooth simple Harnack curve is connected, it is a smooth irreducible curve whose contour still equals its amoeba boundary.  Nevertheless, a generic interior logarithmic fiber now contains $2|\det A|$ points instead of two, and its normalized area ratio is divided by $|\det A|^2$.  Since logarithmic fiber multiplicity and normalized area are invariant under toric automorphisms, the pullback cannot be torically equivalent to a simple Harnack curve.  This gives a transparent geometric reason why boundary equality alone cannot recover Harnack rigidity.

The construction leads immediately to two arithmetic questions.  One must decide whether a prescribed Newton polygon is a nonunimodular integral image of a smaller lattice polygon, and one must determine whether the inverse image of the base curve is connected.  Both questions are resolved here by exact lattice calculations.  If $\Delta$ has vertices $v_0,\ldots,v_s$, the subgroup generated by the differences $v_j-v_0$ has finite index $\kappa(\Delta)$ in $\Z^2$.  Smith normal form shows that $\kappa(\Delta)$ is the greatest common divisor of the nonzero $2\times2$ determinants formed by the vertex differences; the necessary integral-matrix computations are standard consequences of the theory developed by Newman and Cohen \cite{Newman,Cohen}.  The polygon admits a nontrivial lattice quotient exactly when $\kappa(\Delta)>1$, and all quotient polygons are obtained from intermediate lattices, which may be enumerated uniquely by Hermite normal forms.  This translates a geometric covering problem into a finite and executable arithmetic classification.

Connectedness requires more than the Newton polygon.  If $H\subset(\C^*)^2$ is connected, the inclusion induces a monodromy lattice $\Lambda_H\subset\Z^2$.  Covering-space theory identifies the connected components of $\pi_A^{-1}(H)$ with the monodromy orbits in $\Z^2/A\Z^2$ \cite{Hatcher}.  The pullback is connected precisely when $\Lambda_H+A\Z^2=\Z^2$, and its exact number of components is the index $[\Z^2:\Lambda_H+A\Z^2]$.  Toric boundary loops provide computable elements of $\Lambda_H$: for a Newton-nondegenerate curve, the primitive outward normals of the polygon sides arise as winding vectors around the toric ends \cite{CoxLittleSchenck}.  Their maximal minors give a fast sufficient connectedness test.  When that test is inconclusive, compact cycles must be included, and the full monodromy lattice supplies the definitive answer.  This separation between polygon arithmetic, end monodromy, and compact-cycle monodromy is essential for avoiding false irreducibility claims.

The general covering theorem is accompanied by elementary families in which every geometric feature can be seen directly.  The curves $1+z^n+w=0$ are connected degree-$n$ pullbacks of a line.  Their amoebas admit an exact description by triangle inequalities, their contours are smooth and equal to their boundaries, and their generic logarithmic fibers contain $2n$ points.  A richer graph family is defined by $(z^m-a)w-(z^m+b)(z^m+c)=0$, with $a,b,c>0$ and $b\neq c$.  On every circle in the base variable, direct angular differentiation gives exactly one maximum and one minimum of the second logarithmic coordinate.  These extrema form the complete boundary, and the monomial pullback produces an infinite family of irreducible non-Harnack curves with boundary-only contours.  The calculations illustrate in their simplest form the principle that the contour may remain geometrically perfect while the logarithmic map acquires many additional sheets.

Finite coverings do not explain every example.  A polygon with $\kappa(\Delta)=1$ is monomially primitive and cannot be obtained as a nonunimodular integral image of a smaller lattice polygon.  The primitive trapezoid studied here therefore requires a different construction.  For the family $f_{p,q}(z,w)=1+z+pz^2+w+qzw$, a Gaussian-rational parameter box is treated by exact elimination and outward-rounded rational interval arithmetic.  The equation is linear in $w$, so the logarithmic map reduces to the angular behavior of the rational function $-(1+z+pz^2)/(1+qz)$.  A uniform harmonic-dominance estimate proves that every circle $|z|=r$ has exactly two simple angular critical points, one maximum and one minimum.  This produces a smooth critical locus, a smooth embedded contour, and the exact identity $\Cont(\A_{f_{p,q}})=\partial\A_{f_{p,q}}$ throughout the parameter box.  Independent coefficient phases obstruct every toric real normalization, so none of these curves is torically equivalent to a simple Harnack curve.  This primitive result strengthens the covering examples and confirms that the failure of contour--boundary rigidity is not an artifact of a hidden quotient lattice \cite{NisseRigidity}.

The paper also examines positive-genus curves whose Newton polygons contain interior lattice points.  Khovanskii's formula identifies the genus of a smooth Newton-nondegenerate compactification with the number of interior lattice points of its Newton polygon \cite{Kho78}.  In the explicit genus-one and genus-three investigations, discriminants of quadratic presentations certify irreducibility and smoothness, while lopsidedness and the order map locate the bounded complement components.  Affine relations among support exponents produce phase products invariant under scalar multiplication, torus translations, and unimodular monomial changes.  Nonreal values of these invariants exclude toric equivalence to a real curve and hence to a simple Harnack curve.  The genus-three example has three bounded complement components, one for each interior lattice point of its pentagonal Newton polygon.  Its computed contour follows the boundary through extremely narrow amoeba strips, illustrating both the strength and the limitation of numerical contour evidence.  These later examples are presented with an explicit distinction between exact algebraic certification and numerical verification.

Several conceptual conclusions emerge.  The identity between contour and boundary is an assertion about the image of the critical locus, not about the number of critical lifts producing that image.  Smoothness of the contour does not imply injectivity of the critical parametrization, because several source branches may cover the same smooth image germ.  Nonreal coefficients do not themselves exclude Harnack equivalence, since a unitary torus translation can make a real equation appear nonreal without changing its amoeba.  Conversely, lattice-index, monodromy, fiber-multiplicity, area, and phase invariants survive the relevant toric operations and therefore provide effective obstructions.  The correct rigidity problem must take these hidden data into account.

From this viewpoint, the relation with the challenge of Lang, Shapiro, and Shustin is especially clear.  Their work asks for effective control of contour regularity and for quantitative relations between the topology of the amoeba boundary and the singularity structure of the remaining contour \cite{LSS2021}.  The constructions in this paper occupy the extremal regime in which the remaining internal contour is empty.  They show that even this strongest set-theoretic simplification does not identify the algebro-geometric nature of the curve.  Boundary topology, critical-image regularity, logarithmic multiplicity, normalized area, exponent-lattice structure, monodromy, and coefficient phases are related but distinct layers of information.

The organization of the paper follows this progression.  The opening section develops toric equivalence and transport under finite monomial coverings, proves the general pullback theorem, and gives the exact triangular family.  The next section treats the trapezoidal graph family by a direct angular argument.  The lattice-theoretic section classifies all quotient polygons, establishes the monodromy criterion for connectedness, and gives the uniform primitive-trapezoid certificate.  The final sections investigate positive-genus examples with bounded amoeba-complement components and phase obstructions to Harnack equivalence.  Together, these results replace the proposed contour--boundary rigidity by a more precise principle: an amoeba boundary can expose every logarithmic critical value while concealing both the multiplicity and the phase geometry of the curve that produces it.

\medskip

\noindent{\it Acknowledgements.} The author is deeply grateful to Boris Shapiro for kindly sharing the joint paper with Lionel Lang and Eugenii Shustin \cite{LSS2021}; the present work was inspired by the challenging question raised in that paper.
Also, I am gratefully acknowledges the use of Codex AI, based on GPT-6, which used 
 Python, NumPy, SciPy, Matplotlib, Pillow, and LaTeX with TikZ and PGFPlots for the exact computations, numerical experiments, and figures presented in this work.  


\section{Amoebas, contours, and toric equivalence}

For a Laurent polynomial $f\in\C[z^{\pm1},w^{\pm1}]$, assume that $C=V(f)$ is smooth in $(\C^*)^2$.  The logarithmic critical locus is $S_f=\Crit(\Log|_C)$.  It is described by the logarithmic Gauss map $\gamma_f:C\to\mathbb P^1(\C)$, $\gamma_f(z,w)=[zf_z:wf_w]$: one has $S_f=\gamma_f^{-1}(\mathbb P^1(\R))$.  This criterion concerns critical points on the source curve, whereas $\Cont(\A_f)=\Log(S_f)$ is their critical-value image.  Distinct critical points can share one logarithmic value, and a contour point can lie inside $\A_f$.

A toric automorphism is a composition of a torus translation and a unimodular monomial map.  It has the form $T(z,w)=(c_1z^aw^b,c_2z^cw^d)$ with $c_1,c_2\in\C^*$ and $U=\left(\begin{smallmatrix}a&b\\c&d\end{smallmatrix}\right)\in\operatorname{GL}(2,\Z)$.  Its logarithmic action is $\Log\circ T=U\circ\Log+(\log|c_1|,\log|c_2|)$.  Consequently it preserves the cardinality of corresponding logarithmic fibers and the normalized ratio $\Area(\A_f)/(\pi^2\Area(\Newt(f)))$.

This observation also removes an ambiguity concerning nonreal coefficients.  A unitary torus translation can turn a real defining equation into a nonreal one without changing the amoeba.  The resulting curve remains torically equivalent to the original real curve.  Therefore nonreality of displayed coefficients is not, by itself, an obstruction to Harnack equivalence.  Fiber multiplicity, normalized amoeba area, or a phase invariant that survives toric automorphisms is required.

\subsection{Transport under a finite monomial covering}

Let $A=(a_{jk})\in\operatorname{Mat}_{2\times2}(\Z)$ with $d=|\det A|>0$, and define $\pi_A(z,w)=(z^{a_{11}}w^{a_{12}},z^{a_{21}}w^{a_{22}})$.  This is an unramified covering of degree $d$.  The exact identity $\Log\circ\pi_A=A\circ\Log$ will be used repeatedly.

\begin{lemma}\label{lem:transport}
Let $H\subset(\C^*)^2$ be a smooth curve and let $C=\pi_A^{-1}(H)$.  Then $C$ is smooth and reduced, possibly disconnected, and
$$
\A_C=A^{-1}(\A_H),\qquad \partial\A_C=A^{-1}(\partial\A_H),\qquad
S_C=\pi_A^{-1}(S_H),\qquad \Cont(\A_C)=A^{-1}(\Cont(\A_H)).
$$
\end{lemma}

\begin{proof}
The covering is \emph{\'etale}, so the inverse image of the smooth reduced divisor $H$ is smooth and reduced.  If $p\in C$, then $\Log(\pi_A(p))=A\Log(p)$, which gives $\A_C\subseteq A^{-1}(\A_H)$.  Conversely, choose $x\in A^{-1}(\A_H)$ and $q\in H$ with $\Log(q)=Ax$.  The induced map on argument tori is a surjective finite covering.  Hence $q$ has a lift $p$ whose moduli are prescribed by $x$, so $\Log(p)=x$.  This proves the amoeba identity.  Since $A$ is a homeomorphism of $\R^2$, it commutes with taking topological boundaries.

Differentiation along $C$ gives $d(\Log|_H)_{\pi_A(p)}\circ d(\pi_A|_C)_p=A\circ d(\Log|_C)_p$.  Both $d\pi_A$ and $A$ are invertible.  Thus $d(\Log|_C)$ loses rank at $p$ if and only if $d(\Log|_H)$ loses rank at $\pi_A(p)$.  This proves the critical-locus identity, and applying $\Log$ proves the contour identity.
\end{proof}

\begin{proposition}\label{prop:fiber-area}
For corresponding regular values $x$ and $Ax$, one has $\#(C\cap\Log^{-1}(x))=d\,\#(H\cap\Log^{-1}(Ax))$.  Moreover, $\Area(\A_C)=\Area(\A_H)/d$, while $\Area(\Newt(C))=d\Area(\Newt(H))$.
\end{proposition}

\begin{proof}
Every point of $H\cap\Log^{-1}(Ax)$ has exactly $d$ lifts under $\pi_A$.  The kernel of the covering is contained in $(S^1)^2$, so all lifts have logarithmic value $x$.  This proves the fiber formula.  The linear map $A$ has Jacobian determinant of absolute value $d$, and Lemma~\ref{lem:transport} gives the amoeba-area formula.  Exponents transform by $A^{\mathsf T}$, so $\Newt(C)=A^{\mathsf T}\Newt(H)$ up to a lattice translation; its area is multiplied by $d$ \cite{GKZ,CoxLittleSchenck}.
\end{proof}

\begin{theorem}\label{thm:general-cover}
Let $H$ be a smooth Newton-nondegenerate simple Harnack curve with two-dimensional Newton polygon $\Delta$, and assume that its contour is a smooth embedded real-analytic curve.  Let $A$ be an integral matrix with $d=|\det A|>1$.  If $C=\pi_A^{-1}(H)$ is connected, then $C$ is smooth, irreducible, and Newton nondegenerate, its contour is smooth and embedded, and $\Cont(\A_C)=\partial\A_C$.  The curve $C$ is not torically equivalent to a simple Harnack curve.  Its generic interior logarithmic fiber has $2d$ points and
$$
\frac{\Area(\A_C)}{\pi^2\Area(\Newt(C))}=\frac1{d^2}.
$$
\end{theorem}

\begin{proof}
Smoothness follows from Lemma~\ref{lem:transport}.  A connected smooth algebraic curve is irreducible because distinct irreducible components of a smooth curve are disjoint.  Newton nondegeneracy is preserved: every face truncation of the pullback is the pullback of a face truncation of the base polynomial, and its logarithmic derivative vector is transformed by the invertible real matrix $A$.

For a simple Harnack curve, the complete contour equals the amoeba boundary and a generic interior logarithmic fiber has two points \cite{Mikhalkin2000,MR2001}.  Lemma~\ref{lem:transport} gives $\Cont(\A_C)=A^{-1}(\Cont(\A_H))=A^{-1}(\partial\A_H)=\partial\A_C$.  The same real linear diffeomorphism preserves smoothness and embeddedness of the image.  Proposition~\ref{prop:fiber-area} gives the multiplicity $2d$.  It also gives $\Area(\A_C)=\pi^2\Area(\Delta)/d$ and $\Area(\Newt(C))=d\Area(\Delta)$, yielding the displayed ratio.  A simple Harnack curve has generic multiplicity two and normalized area ratio one.  Since both quantities are invariant under toric equivalence, $C$ cannot be torically equivalent to one.
\end{proof}

Since smooth simple Harnack curves exist for every two-dimensional lattice polygon through toric patchworking \cite{Mikhalkin2000}, Theorem~\ref{thm:general-cover} produces examples over infinitely many integral images $A^{\mathsf T}\Delta$.  Connectedness cannot be omitted, and it will be treated exactly in Section~\ref{sec:connectedness}.

\subsection{The exact triangular family}

For every integer $n\geq2$, let $f_n(z,w)=1+z^n+w$.  Its Newton polygon is $\Delta_n=\operatorname{conv}\{(0,0),(n,0),(0,1)\}$.

\begin{theorem}\label{thm:triangle-family}
The curve $C_n=V(f_n)$ is smooth, connected, irreducible, and Newton nondegenerate.  Its complete finite contour is smooth and embedded, $\Cont(\A_{f_n})=\partial\A_{f_n}$, a generic interior logarithmic fiber has $2n$ points, and
$$
\Area(\A_{f_n})=\frac{\pi^2}{2n},\qquad
\frac{\Area(\A_{f_n})}{\pi^2\Area(\Delta_n)}=\frac1{n^2}.
$$
In particular, $C_n$ is not torically equivalent to a simple Harnack curve.
\end{theorem}

\begin{proof}
The equation is linear in $w$, so $C_n$ is the graph $w=-1-z^n$ over $\C^*$ with the $n$ points $z^n=-1$ removed.  It is connected and irreducible.  Since $(f_n)_w=1$, it is smooth.  Every edge truncation is a monomial or a binomial with simple nonzero roots, proving Newton nondegeneracy.

Put $x=\log|z|$ and $y=\log|w|$.  For fixed $|z|=e^x$, the number $u=z^n$ has modulus $e^{nx}$, and $|1+u|$ fills the interval $[|1-e^{nx}|,1+e^{nx}]$.  Hence the exact amoeba criterion is $|1-e^{nx}|\leq e^y\leq1+e^{nx}$.  The upper boundary is $y=\log(1+e^{nx})$, while the two lower branches are $y=\log|1-e^{nx}|$ on $x<0$ and $x>0$.

The logarithmic Gauss map is $\gamma_{f_n}=[nz^n:w]$.  On $w=-(1+z^n)$, criticality is equivalent to $z^n/(1+z^n)\in\R$.  If $z^n=a+ib$, then $\operatorname{Im}(z^n/(1+z^n))=b/((1+a)^2+b^2)$.  Thus the critical locus is given by $\operatorname{Im}(z^n)=0$.  Positive values of $z^n$ produce the upper boundary; negative values produce the two lower branches.  No finite internal contour branch occurs, so the contour equals the boundary and is smooth and embedded.

At a generic interior point, the strict triangle inequalities give two possible arguments of $u$, related by conjugation.  Each nonzero $u$ has $n$ distinct roots $z$, and $w$ is then fixed, giving $2n$ lifts.  The curve is the pullback of $1+u+w=0$ by $(z,w)\mapsto(z^n,w)$.  The base amoeba has area $\pi^2/2$ \cite{MR2001,PassareRullgard}; the inverse logarithmic map divides area by $n$.  Since $\Area(\Delta_n)=n/2$, the normalized ratio is $1/n^2$.
\end{proof}

\section{A trapezoidal graph family and an exact contour proof}

Fix $a,b,c>0$ with $b\neq c$, let $m>1$, and define $f_{m,a,b,c}(z,w)=(z^m-a)w-(z^m+b)(z^m+c)$.  Its Newton polygon is $\Delta_m=\operatorname{conv}\{(0,0),(2m,0),(m,1),(0,1)\}$.

\begin{theorem}\label{thm:graph-family}
The curve $C_{m,a,b,c}=V(f_{m,a,b,c})$ is smooth, irreducible, and Newton nondegenerate.  Its critical locus and contour are smooth, the contour is embedded, $\Cont(\A_{f_{m,a,b,c}})=\partial\A_{f_{m,a,b,c}}$, and the generic logarithmic fiber has $2m$ points.  Hence the curve is not torically equivalent to a simple Harnack curve.
\end{theorem}

\begin{proof}
Begin with $h_{a,b,c}(u,w)=(u-a)w-(u+b)(u+c)$, parametrized by $u=t$ and $w=(t+b)(t+c)/(t-a)$.  Write $t=re^{i\theta}$.  Away from a zero or pole, direct differentiation gives
$$
\frac{\partial}{\partial\theta}\log\left|\frac{(t+b)(t+c)}{t-a}\right|
=-r\sin\theta\left(\frac{b}{|t+b|^2}+\frac{c}{|t+c|^2}+\frac{a}{|t-a|^2}\right).
$$
The factor in parentheses is strictly positive.  Thus every regular circle $|t|=r$ contains exactly two simple logarithmic critical points, at $\theta=0$ and $\theta=\pi$.  One gives the unique maximum and the other the unique minimum of the second logarithmic coordinate.  The vertical amoeba slice over $\log r$ is exactly the interval between these two extrema.  Its endpoints are all the critical values on the slice, so the complete contour of $V(h_{a,b,c})$ equals its boundary.  The extremal branches are disjoint graphs over $\log r$, and their toric-end limits at the zeros $-b,-c$ and pole $a$ are smooth.

The covering $\pi_m(z,w)=(z^m,w)$ pulls $V(h_{a,b,c})$ back to $C_{m,a,b,c}$.  The two coefficients in the equation linear in $w$ are relatively prime, so the polynomial is irreducible and the pullback is connected.  A singular point would require $z^m=a$, but substitution gives $-(a+b)(a+c)\neq0$.  The horizontal edge polynomials have simple roots because $a,b,c>0$ and $b\neq c$, and the remaining edge truncations are binomials.  Hence the curve is Newton nondegenerate.  The transport theorem proves the contour statement and gives the multiplicity $2m$.
\end{proof}

For the figure we take $m=3$, $a=3/10$, $b=3$, and $c=10$.  The sky-blue region is rendered as a dense family of translucent fixed-argument logarithmic traces rather than as a uniform fill.  Dark red records the complete contour, which Theorem~\ref{thm:graph-family} proves to be exactly the boundary.

\begin{figure}[h]
\centering
\includegraphics[width=.56\textwidth]{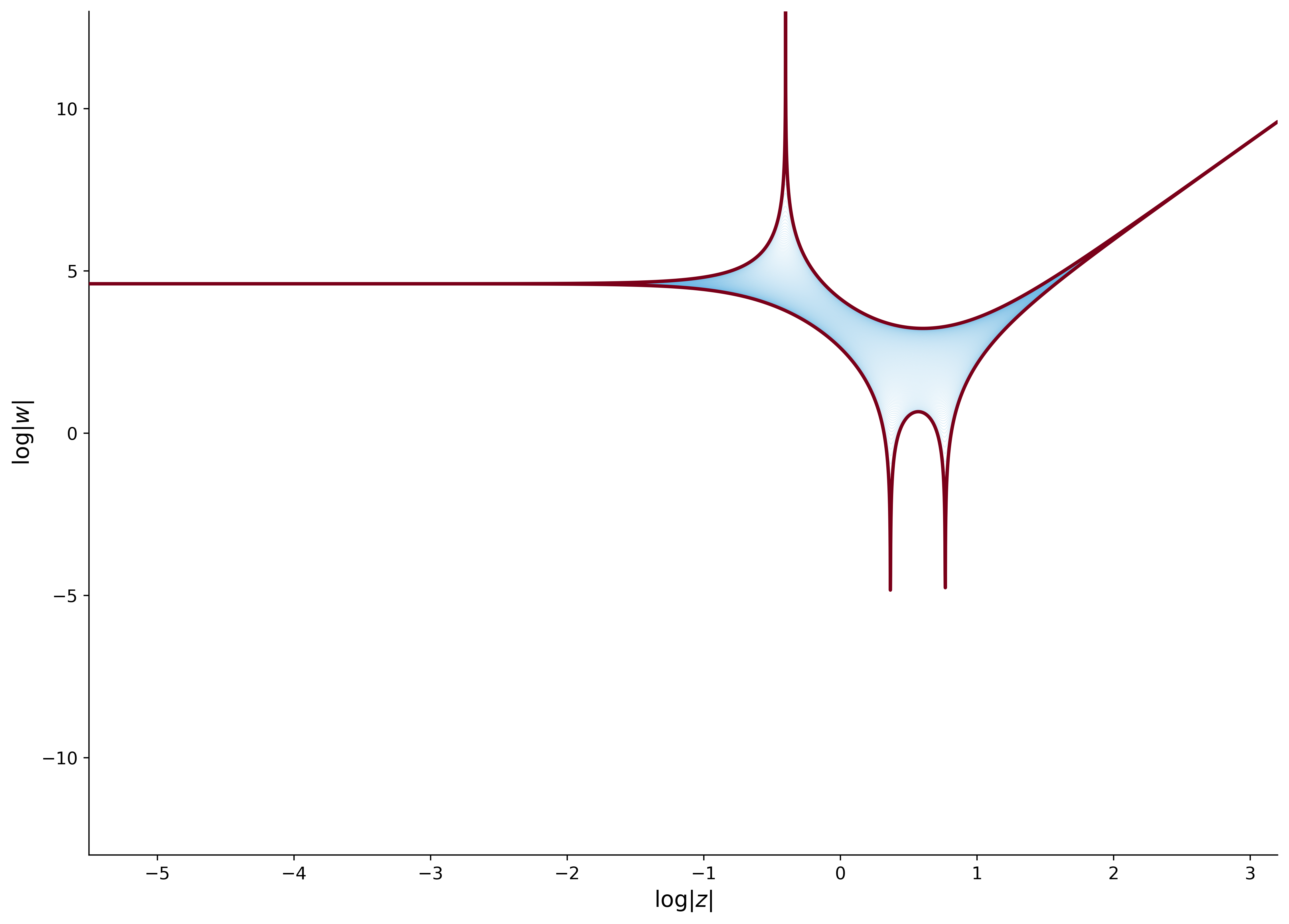}
\caption{The non-flat sky-blue amoeba of $(z^3-3/10)w-(z^3+3)(z^3+10)=0$ and its dark-red complete contour, drawn in an enlarged logarithmic window so that every tentacular direction is visible.  Every red branch is boundary-producing.}
\label{fig:amoeba-contour}
\end{figure}

\section{The lattice index and every quotient polygon}

Let $\Delta\subset\R^2$ be a two-dimensional lattice polygon with vertices $v_0,v_1,\ldots,v_s$.  Define the vertex-difference lattice $\Lambda_\Delta=\sum_{j=1}^s\Z(v_j-v_0)\subset\Z^2$ and its index $\kappa(\Delta)=[\Z^2:\Lambda_\Delta]$.  This index does not depend on the chosen base vertex.  Smith normal form gives the exact formula
$$
\kappa(\Delta)=\gcd_{1\leq i<j\leq s}|\det(v_i-v_0,v_j-v_0)|,
$$
with zero determinants omitted \cite{Newman,Cohen}.

\begin{theorem}\label{thm:arithmetic}
There exist a lattice polygon $\Delta_0$, a vector $t\in\Z^2$, and an integral matrix $A$ with $|\det A|>1$ such that $\Delta=t+A^{\mathsf T}\Delta_0$ if and only if $\kappa(\Delta)>1$.  All such factorizations correspond exactly to proper rank-two intermediate lattices $L$ satisfying $\Lambda_\Delta\subseteq L\subsetneq\Z^2$.
\end{theorem}

\begin{proof}
If $\Delta=t+A^{\mathsf T}\Delta_0$, then every vertex difference belongs to the proper lattice $A^{\mathsf T}\Z^2$, hence $\Lambda_\Delta$ is proper and $\kappa(\Delta)>1$.  Conversely, suppose $\kappa(\Delta)>1$ and choose a basis matrix $B$ for any intermediate lattice $L$.  Since every $v_j-v_0$ lies in $B\Z^2$, the polygon $\Delta_0=B^{-1}(\Delta-v_0)$ has integral vertices.  Taking $A^{\mathsf T}=B$ gives the factorization, whose covering degree is $[\Z^2:L]=|\det B|$.  These two constructions are inverse to each other.
\end{proof}

Every intermediate lattice can be enumerated without repetition by column Hermite normal form.  For every divisor $N>1$ of $\kappa(\Delta)$, examine the matrices $B=\left(\begin{smallmatrix}a&b\\0&d\end{smallmatrix}\right)$ with $ad=N$ and $0\leq b<a$.  Retain precisely those for which every vertex difference belongs to $B\Z^2$.  If a difference is $(x,y)$, membership is equivalent to $d\mid y$ and $a\mid x-b(y/d)$.  The quotient polygon is $B^{-1}(\Delta-v_0)$ and the covering matrix is $A=B^{\mathsf T}$.  Hermite uniqueness proves that no quotient lattice is omitted or repeated.

The polygon is called monomially primitive when $\kappa(\Delta)=1$.  Such a polygon cannot arise as a nonunimodular integral image of a smaller lattice polygon.  The primitive trapezoid $\operatorname{conv}\{(0,0),(2,0),(1,1),(0,1)\}$ has index one because $\det((1,1),(0,1))=1$.

\subsection{Connectedness and toric-end monodromy}\label{sec:connectedness}

Let $i:H\hookrightarrow(\C^*)^2$ be a connected smooth curve and identify $\pi_1((\C^*)^2)$ with $\Z^2$.  Since the target group is abelian, define $\Lambda_H=i_*H_1(H,\Z)\subseteq\Z^2$.  The covering subgroup of $\pi_A$ is $A\Z^2$, and its deck group is $G_A=\Z^2/A\Z^2$.

\begin{theorem}\label{thm:monodromy}
The pullback $\pi_A^{-1}(H)$ is connected if and only if $\Lambda_H+A\Z^2=\Z^2$.  In general, its number of connected components is $[\Z^2:\Lambda_H+A\Z^2]$.
\end{theorem}

\begin{proof}
Label one fiber of the covering by $G_A$.  Lifting a loop in $H$ translates the label by the residue class of its image in $\Z^2$.  The monodromy orbits are therefore the cosets of the image of $\Lambda_H$ in $G_A$.  Each orbit is the intersection of one connected component of the pullback with the chosen fiber.  The action is transitive exactly when $\Lambda_H\to G_A$ is surjective, which is equivalent to the displayed lattice equality.  The number of orbits is the displayed index, as in the standard classification of covering spaces \cite{Hatcher}.
\end{proof}

Assume that $H$ is Newton nondegenerate with polygon $\Delta_0$.  For each edge $E$, let $n_E$ be a primitive outward normal.  A small loop around a transverse intersection with the toric divisor corresponding to $E$ maps, up to sign, to $n_E$ in $H_1((\C^*)^2,\Z)$ \cite{CoxLittleSchenck}.  Hence $N_{\Delta_0}=\sum_E\Z n_E$ is contained in $\Lambda_H$.

\begin{proposition}\label{prop:end-test}
If $N_{\Delta_0}+A\Z^2=\Z^2$, then $\pi_A^{-1}(H)$ is connected.  The condition is checked by forming the integer matrix whose columns are the primitive edge normals and the two columns of $A$; it holds exactly when the greatest common divisor of all its $2\times2$ minors is one.
\end{proposition}

\begin{proof}
The inclusion $N_{\Delta_0}\subseteq\Lambda_H$ shows that the hypothesis implies $\Lambda_H+A\Z^2=\Z^2$.  Theorem~\ref{thm:monodromy} gives connectedness.  For any rank-two integer matrix, the gcd of its maximal minors is the index of the lattice generated by its columns, by Smith normal form \cite{Newman,Cohen}.
\end{proof}

The end test is sufficient but not necessary.  Compact cycles in $H$ can enlarge $N_{\Delta_0}$ to a larger subgroup $\Lambda_H$ and can make the full monodromy surjective even when the end normals alone do not.  Therefore failure of the gcd test is an inconclusive verdict, not a proof that the pullback is disconnected.  Exact connectedness in such a case requires the full lattice $\Lambda_H$, obtained from an explicit homology basis or a certified continuation graph.  Appendix~\ref{app:algorithm} gives an executable implementation that extracts all compact-cycle winding vectors from such a graph and then computes the exact number $[\Z^2:\Lambda_H+A\Z^2]$ of pullback components.

\subsection{The primitive trapezoid}\label{sec:primitive}

Consider $f_{p,q}(z,w)=1+z+pz^2+w+qzw=P_p(z)+Q_q(z)w$, where $P_p(z)=1+z+pz^2$ and $Q_q(z)=1+qz$.  Its Newton polygon is $\Delta=\operatorname{conv}\{(0,0),(2,0),(1,1),(0,1)\}$.  The support differences contain $(1,0)$ and $(0,1)$, hence $\kappa(\Delta)=1$.  Theorem~\ref{thm:arithmetic} proves that no nonunimodular Newton-polygon quotient can account for this construction.

Put $p_0=1/100-i/5000$, $q_0=1/20-7i/200$, and $\varepsilon=3/50000000$.  Let $\mathcal B$ be the closed rational box defined by $|\operatorname{Re}p-1/100|\leq\varepsilon$, $|\operatorname{Im}p+1/5000|\leq\varepsilon$, $|\operatorname{Re}q-1/20|\leq\varepsilon$, and $|\operatorname{Im}q+7/200|\leq\varepsilon$.

\begin{theorem}\label{thm:primitive}
For every $(p,q)\in\mathcal B$, the curve $C_{p,q}=V(f_{p,q})\subset(\C^*)^2$ is smooth and irreducible, and its toric compactification is smooth and Newton nondegenerate.  Its logarithmic critical locus is smooth, its complete contour is a smooth embedded real-analytic curve, and $\Cont(\A_{f_{p,q}})=\partial\A_{f_{p,q}}$.  Every regular logarithmic fiber has at most two points.  No member of this box is torically equivalent to a smooth simple Harnack curve.
\end{theorem}

The proof below expands the elimination and rational interval  used in \cite{NisseRigidity}.  
%
%
The equation is linear in $w$ and gives $w=-P_p(z)/Q_q(z)$.  The exact resultant is $\operatorname{Res}_z(P_p,Q_q)=p+q^2-q$, up to the harmless conventional power of the leading coefficient.  Writing $p=u+iv$ and $q=a+ib$, direct outward-rounded rational interval evaluation on $\mathcal B$ gives
$$
u\in\left[\frac{499997}{50000000},\frac{500003}{50000000}\right],\quad
v\in\left[-\frac{10003}{50000000},-\frac{9997}{50000000}\right],
$$
$$
a\in\left[\frac{2499997}{50000000},\frac{2500003}{50000000}\right],\quad
b\in\left[-\frac{1750003}{50000000},-\frac{1749997}{50000000}\right].
$$
For the resultant, the same exact calculation yields
$$
\operatorname{Re}(p+q^2-q)\in
\left[-\frac{193625591}{5000000000},-\frac{193624409}{5000000000}\right]
$$
and
$$
\operatorname{Im}(p+q^2-q)\in
\left[\frac{39124852250009}{1250000000000000},
\frac{39125147750009}{1250000000000000}\right].
$$
In particular, $P_p$ and $Q_q$ have no common root.  The discriminant $1-4p$ satisfies
$$
\operatorname{Re}(1-4p)\in
\left[\frac{11999997}{12500000},\frac{12000003}{12500000}\right],\qquad
\operatorname{Im}(1-4p)\in
\left[\frac{9997}{12500000},\frac{10003}{12500000}\right].
$$
Thus $P_p$ has two simple nonzero roots, $Q_q$ has one simple nonzero root, and they are relatively prime.  If $f=f_z=f_w=0$, then $Q_q(z)=P_p(z)=0$, which is impossible.  Hence the affine curve is smooth.  A factorization of the primitive polynomial $P_p+Q_qw$, which has degree one in $w$, would force a common nonconstant divisor of $P_p$ and $Q_q$; therefore it is irreducible.  The four face truncations have simple torus roots because $p$, $q$, $1-4p$, and $p+q^2-q$ avoid zero.  This proves Newton nondegeneracy and toric transversality by the standard face criterion \cite{GKZ,CoxLittleSchenck}.

\subsection{Real critical polynomial and elimination ideals}

For reference, put $z=x+iy$, $w=s+it$, $p=u+iv$, and $q=a+ib$.  Write $F_1=\operatorname{Re}f_{p,q}$ and $F_2=\operatorname{Im}f_{p,q}$.  Their completely expanded forms are
$$
F_1=1+x+u(x^2-y^2)-2vxy+s+a(xs-yt)-b(xt+ys)
$$
and
$$
F_2=y+v(x^2-y^2)+2uxy+t+b(xs-yt)+a(xt+ys).
$$
The logarithmic Gauss condition is the polynomial equation $K=\operatorname{Im}(z f_z\overline{w f_w})=0$, where $f_z=1+2pz+qw$ and $f_w=1+qz$.  Thus the affine logarithmic critical source is the real algebraic set $F_1=F_2=K=0$ saturated by $(x^2+y^2)(s^2+t^2)$.

Let $J$ be the $3\times4$ Jacobian matrix of $(F_1,F_2,K)$ with columns ordered by $(x,y,s,t)$, and let $\nu_j=(-1)^{j+1}\det J_{\widehat j}$.  At a smooth critical point, $\nu=(\nu_1,\nu_2,\nu_3,\nu_4)$ spans the tangent line.  Clearing the positive logarithmic denominators gives the image tangent
$$
V_1=(x\nu_1+y\nu_2)(s^2+t^2),\qquad
V_2=(s\nu_3+t\nu_4)(x^2+y^2).
$$
The saturated ramification ideal is
$$
I_{\mathrm{ram}}=\bigl(F_1,F_2,K,V_1,V_2\bigr):
\bigl((x^2+y^2)(s^2+t^2)I_3(J)\bigr)^\infty,
$$
where $I_3(J)$ denotes the ideal of maximal minors.  Empty real locus of this ideal in all toric charts excludes a zero tangent of the contour image.

For the off-diagonal image test, introduce a second lift $(X,Y,S,T)$ and the corresponding primed equations.  Equality of logarithmic images is
$$
E_z=x^2+y^2-X^2-Y^2=0,qquad E_w=s^2+t^2-S^2-T^2=0.
$$
With $D=(x-X)^2+(y-Y)^2+(s-S)^2+(t-T)^2$ and
$$
U=(x^2+y^2)(s^2+t^2)(X^2+Y^2)(S^2+T^2),
$$
the equations $\sigma D-1=0$ and $\tau U-1=0$ remove the ordinary diagonal and all coordinate hyperplanes.  The saturated two-lift ideal is generated by $F_1,F_2,K,F'_1,F'_2,K',E_z,E_w,\sigma D-1,\tau U-1$.  For the present family the angular certificate below proves that its real locus is empty.  This is stronger than testing finitely many sampled pairs and avoids any ambiguity from close numerical images.

\subsection{Angular elimination and the rational interval inequalities}

Write $z=re^{i\theta}$.  On the graph, the logarithmic map is $(r,\theta)\mapsto(\log r,Y_r(\theta))$, where $Y_r(\theta)=\log|P_p(re^{i\theta})|-\log|Q_q(re^{i\theta})|$.  Put
$$
A(z)=z(P_p'Q_q-P_pQ_q')=(1-q)z+2pz^2+pqz^3
$$
and $B(z)=P_p(z)Q_q(z)=1+(1+q)z+(p+q)z^2+pqz^3$.  Away from the zeros and pole, the exact critical equation is $N_r(\theta)=\operatorname{Im}(A(re^{i\theta})\overline{B(re^{i\theta})})=0$.

Write $A(z)=\sum_{j=1}^3a_jz^j$ and $B(z)=\sum_{\ell=0}^3b_\ell z^\ell$, with $(a_1,a_2,a_3)=(1-q,2p,pq)$ and $(b_0,b_1,b_2,b_3)=(1,1+q,p+q,pq)$.  Define
$$
D_k(r)=\sum_{j-\ell=k}a_j\overline{b_\ell}r^{j+\ell},\qquad
C_k(r)=\frac{D_k(r)-\overline{D_{-k}(r)}}{2i}.
$$
Then $N_r(\theta)=C_0(r)+2\operatorname{Re}(C_1(r)e^{i\theta}+C_2(r)e^{2i\theta}+C_3(r)e^{3i\theta})$.  Every real and imaginary part of every $C_k$ is a polynomial in $r$ with rational interval coefficients on $\mathcal B$.

Set $A_1=2|C_1|$, $B_1=|C_0|+2|C_2|+2|C_3|$, and $D_1=4|C_2|+6|C_3|$.  The rational certificate establishes, uniformly for $(p,q)\in\mathcal B$ and $r>0$,
$$
B_1(r)<A_1(r),\qquad D_1(r)^2<A_1(r)^2-B_1(r)^2.
$$
The radial half-line is covered by the two rational charts $0<r\leq1$ and $0<\rho=1/r\leq1$.  In each chart, positive monomial factors are removed, all modulus bounds are squared only after their signs have been certified, and the remaining polynomial interval expressions are evaluated by Horner form with outward rational endpoints.  Boxes for which a direct interval evaluation does not decide the sign are bisected dyadically.  Sturm chains of the endpoint polynomials certify that no missed sign change occurs on a terminal radial interval.  This exact procedure and the resulting box $\mathcal B$ are recorded in \cite{NisseRigidity}; the elementary coefficient exclusions have been reproduced above so that the toric and affine portions of the certificate are independently checkable from the present text.

After rotating the angular variable by $\arg C_1(r)$, the first harmonic is $A_1\cos\theta$ and the remaining harmonics have absolute value at most $B_1$.  Every zero lies in one of two disjoint arcs on which $|\sin\theta|\geq\sqrt{A_1^2-B_1^2}/A_1$.  The derivative of the first harmonic has magnitude at least $\sqrt{A_1^2-B_1^2}$ there, while the derivative of the remaining harmonics has magnitude at most $D_1$.  The second strict inequality forces one sign for the full derivative on each arc.  Hence every circle has exactly two simple angular critical points, one minimum and one maximum.

Along either critical branch, the contour parametrization has first coordinate $\log r$ and therefore derivative $1/r\neq0$.  This excludes ramification.  If two distinct critical lifts had the same logarithmic image, equality of the first coordinate would give the same $r$, but the two critical values at that radius are the distinct minimum and maximum of $Y_r$.  Hence the off-diagonal two-lift real locus is empty.

\subsection{Chambers, toric ends, and phase obstruction}

The vertical amoeba slice over $\log r$ is the interval between the minimum and maximum of $Y_r$.  Its endpoints are precisely the two critical values.  Therefore every critical value is boundary-producing and every boundary point is critical.  This proves $\Cont(\A_{f_{p,q}})=\partial\A_{f_{p,q}}$, while strict separation of the extrema and absence of ramification prove that the contour is embedded and real analytic.  Every regular value inside a slice has two angular lifts, and every value outside has none.

The toric boundary has five intersection points counted with edge multiplicity.  At either simple root $\alpha$ of $P_p$, one has $w=-P_p'(\alpha)(z-\alpha)/Q_q(\alpha)+O((z-\alpha)^2)$.  At the simple root $\beta=-1/q$ of $Q_q$, one has $1/w=-Q_q'(\beta)(z-\beta)/P_p(\beta)+O((z-\beta)^2)$.  At $z=0$, the equation gives $w=-1+O(z)$, and at infinity it gives $w=-(p/q)z+O(1)$.  Every leading coefficient is separated from zero by the rational bounds above, so all end germs are smooth and remain separated throughout $\mathcal B$.

Finally, the coefficients of $1,z,w$ are already positive real.  If a scalar and torus translation made the entire equation real, then the normalized coefficients $p$ and $q$ would both be real.  Their imaginary intervals displayed above do not contain zero.  This affine phase obstruction is invariant under unimodular changes of exponent coordinates, and every simple Harnack curve is real.  Thus no curve in $\mathcal B$ is torically equivalent to a simple Harnack curve \cite{Mikhalkin2000,MR2001}.  This completes the proof of Theorem~\ref{thm:primitive}.

\begin{figure}[h]
\centering
\includegraphics[width=.56\textwidth]{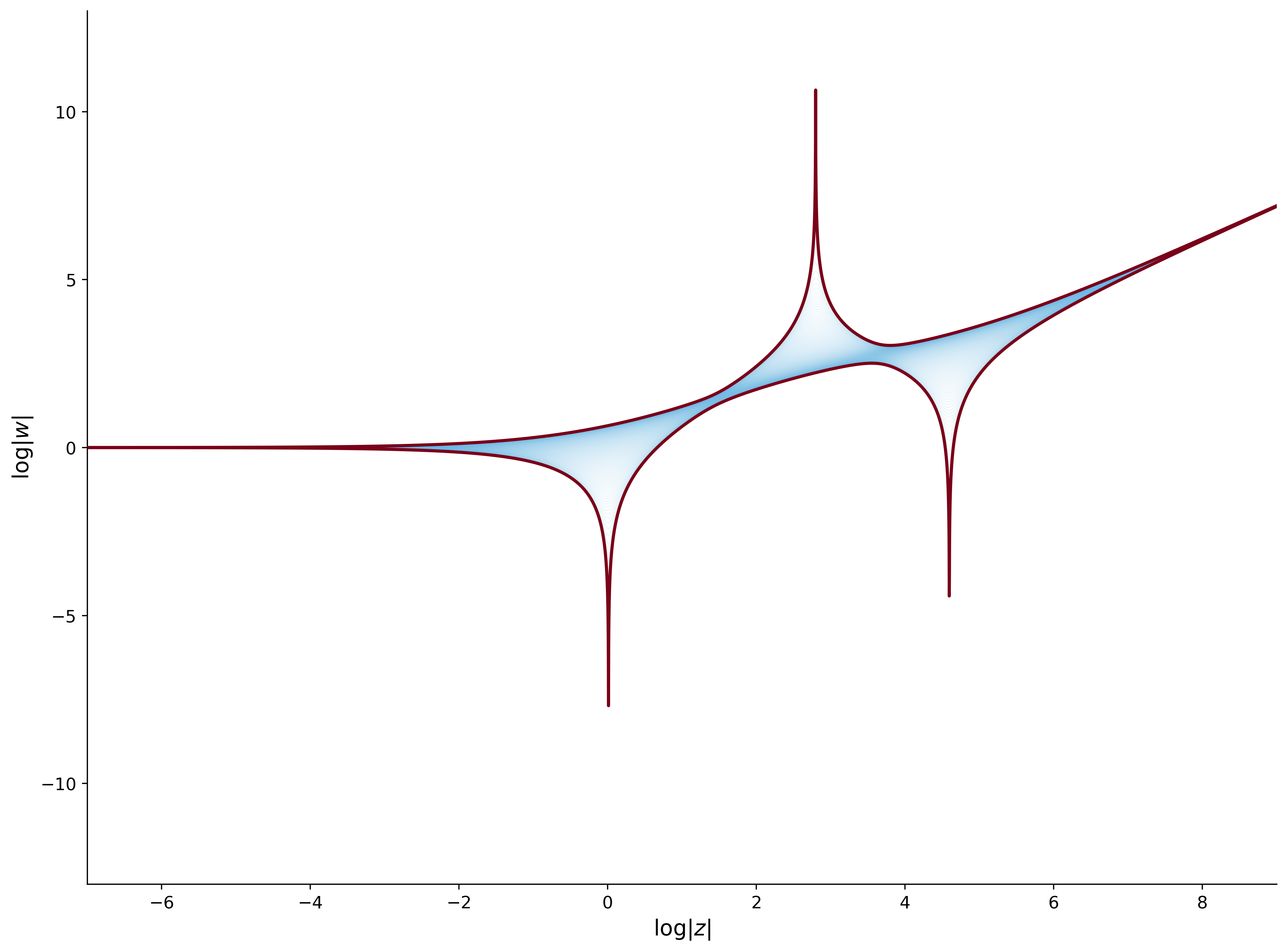}
\caption{The primitive-trapezoid amoeba in a non-flat sky-blue trace texture and its complete contour in dark red, drawn in an enlarged logarithmic window displaying all five tentacular directions.  The two contour branches are the slice minimum and slice maximum, and both are boundary-producing.}
\label{fig:primitive-amoeba}
\end{figure}

This theorem is complementary to the covering theorem rather than a consequence of it.  For another primitive polygon, one must again supply direct coefficient choices and complete critical-source, ramification, two-lift, chamber, and toric-end control.


\begin{theorem}\label{thm:combined}
For every two-dimensional lattice polygon $\Delta$, every smooth Newton-nondegenerate simple Harnack curve with polygon $\Delta$, and every nonunimodular integral matrix whose pullback is connected, the pulled-back curve is smooth, irreducible, and Newton nondegenerate, has a smooth embedded contour satisfying $\Cont(\A)=\partial\A$, and is not torically equivalent to a simple Harnack curve.  Consequently every $\Delta$ generates infinitely many integral-image Newton polygons admitting such curves.  In addition, the primitive trapezoid of Section~8 admits a direct open family with the same contour--boundary property.
\end{theorem}

\begin{proof}
The covering assertion is Theorem~\ref{thm:general-cover}; the connectedness hypothesis is decided by Theorem~\ref{thm:monodromy} and may be certified by Proposition~\ref{prop:end-test}.  Taking matrices of unbounded determinant produces infinitely many polygon areas and therefore infinitely many unimodular equivalence classes.  The primitive assertion is the independently certified construction of \cite{NisseRigidity} described in Section~8.
\end{proof}

The theorem deliberately does not assert that every prescribed primitive polygon admits such a curve.  It gives a general infinite-family result from every starting polygon and one direct primitive open family.  A universal theorem for all primitive polygons would require a uniform direct construction beyond the monomial-covering method.

\subsection{Conclusion}

Contour--boundary equality is not a characterization of the simple-Harnack class.  A finite monomial cover preserves the exposed critical-value geometry but multiplies regular logarithmic sheets.  The normalized area ratio changes from one to $d^{-2}$, and the generic fiber cardinality changes from two to $2d$.  These independent invariants exclude toric Harnack equivalence.

For a prescribed Newton polygon, the exact arithmetic question is settled by $\kappa(\Delta)$ and the intermediate lattices containing $\Lambda_\Delta$.  For each quotient, connectedness is a separate monodromy question governed by $\Lambda_H+A\Z^2$.  Toric-end normals give a fast exact certificate whenever they already generate the deck group.  Primitive polygons lie outside the quotient mechanism and require direct phase constructions such as the trapezoid family.


\section{Example of a genus one curve}

We give an explicit irreducible curve $C\subset(\C^*)^2$ defined by a Laurent polynomial with five terms and Gaussian integer coefficients. It has genus $1$, and its Newton polygon is a quadrilateral without any pair of parallel sides, so it is neither a triangle nor a zonotope. The complement of its amoeba has exactly one bounded component, which is the number of interior lattice points of the polygon. The critical contour of the amoeba map is a smooth curve which, within the accuracy of our computation, coincides with the topological boundary of the amoeba. The curve is nevertheless not toric-equivalent to any simple Harnack curve, because two torus-invariant combinations of its coefficients are not real, and because the area of its amoeba is smaller than $\pi^2$ times the area of the Newton polygon. We display the amoeba in a speckled blue, its contour in crimson, and the Newton polygon, and we say precisely which statements are proved from the cited theorems and which are numerical.

\subsection{The curve and its Newton polygon}

Put $\varepsilon=6+8i$ and let $C\subset(\C^*)^2$ be the zero set of
$f(x,y)=x+y+x^{-1}+x^{-1}y^{-1}-\varepsilon$. Multiplying by $xy$ we obtain the polynomial $F(x,y)=x^2y+xy^2+y+1-\varepsilon\,xy$, which has the same zero set in the torus. The support of $f$ consists of $(1,0)$, $(0,1)$, $(-1,0)$, $(-1,-1)$ and $(0,0)$, hence the Newton polygon $\Delta$ is the convex hull of the first four points. Going around its boundary, the sides are the vectors $(-1,1)$, $(-1,-1)$, $(0,-1)$ and $(2,1)$. Each of them is primitive, the sum of the four is zero, and the cross product of two consecutive ones is $2$, $1$, $2$, $3$ in this cyclic order, so the four points are indeed vertices of a convex quadrilateral. This quadrilateral is not a triangle, and since no two of its sides are parallel it is not a trapezoid, a fortiori not a parallelogram and not centrally symmetric, hence not a zonotope. The shoelace formula gives $\Area(\Delta)=2$, its boundary contains exactly $4$ lattice points because every side is primitive, and Pick's formula gives $2-4/2+1=1$ interior lattice point, namely the origin. The polygon is drawn in Figure~\ref{fig:newton}; the exponents occurring in $f$ are the large red dots.

\begin{figure}[ht]
\centering
\includegraphics[width=0.25\textwidth]{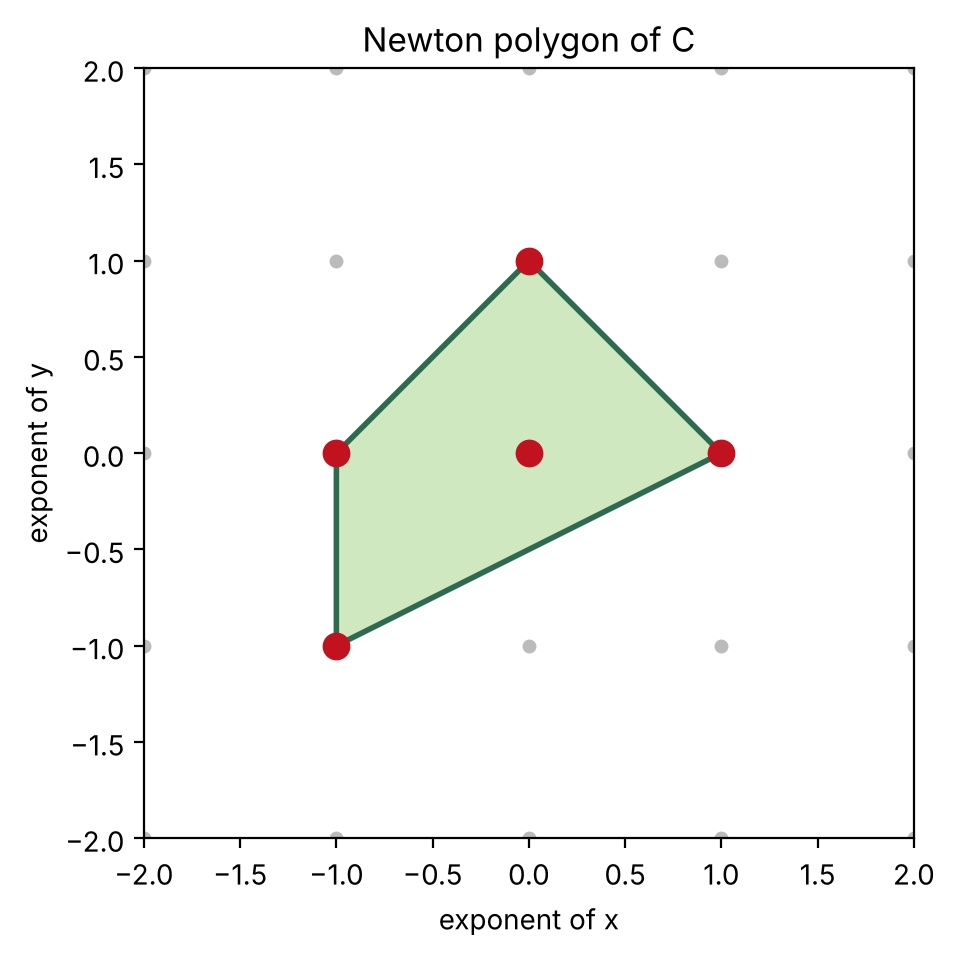}
\caption{The Newton polygon $\Delta$ of $f$: a quadrilateral with vertices $(1,0)$, $(0,1)$, $(-1,0)$, $(-1,-1)$ and the single interior lattice point $(0,0)$. Red dots are the exponents occurring in $f$, grey dots are the other lattice points.}
\label{fig:newton}
\end{figure}

\subsection{Smoothness, irreducibility and genus}

Regard $F=x\,y^2+(x^2-\varepsilon x+1)\,y+1$ as a quadratic polynomial in $y$ over $\C(x)$. Its discriminant is $D(x)=(x^2-\varepsilon x+1)^2-4x$, a polynomial of degree $4$ with $D(0)=1$. A numerical computation gives the four roots $6.50511+7.80945\,i$, $5.38987+8.38795\,i$, $0.05997-0.16254\,i$ and $0.04505-0.03486\,i$, whose mutual distances are at least $0.1285$; so $D$ has four simple roots. Over a value $x_0\in\C^*$ with $D(x_0)\ne0$ the curve has two distinct points, and at those points $\partial_yF\ne0$, so $C$ is smooth there. Over a root $x_0$ of $D$ there is one point $(x_0,y_0)$ with $y_0=-(x_0^2-\varepsilon x_0+1)/(2x_0)$ and $\partial_yF=0$ at it. If $\partial_xF$ vanished as well, $x_0$ would be a multiple root of the discriminant of a quadratic in $y$, which is impossible for a simple root; we also checked directly that $|\partial_xF|$ equals $3.22$, $3.17$, $19.2$ and $57.2$ at the four points. Hence $C$ is smooth in the torus.

If $F$ were reducible, the quadratic in $y$ would have a root in $\C(x)$, because its coefficients $x$, $x^2-\varepsilon x+1$ and $1$ have no common factor, and then $D$ would be a square in $\C(x)$; this is false since $D$ has simple roots. Hence $F$, and therefore $f$, is irreducible.

To compute the genus, put $w=2xy+x^2-\varepsilon x+1$. Then $w^2=D(x)$, and the map $(x,y)\mapsto(x,w)$ identifies $C$ with an open subset of the affine curve $w^2=D(x)$, whose smooth compactification is an elliptic curve because $D$ has degree $4$ and four simple roots (Riemann--Hurwitz, see \cite{Har77}, Chapter~IV, Proposition~2.4: $2g-2=2(-2)+4=0$). It remains to see that the compactification of $C$ does not acquire further branching over $x=0$ and $x=\infty$. Over $x=0$ the two roots of the quadratic are $y\to\infty$ and $y\to-1$, and over $x=\infty$ they are $y\sim-x$ and $y\sim-x^{-2}$; in both cases they are distinct points of $\mathbb{P}^1$, so the double cover is unramified there. Consequently $C$ has geometric genus $1$, which is the number of interior lattice points of $\Delta$. Each side of $\Delta$ carries a binomial with a single simple root, namely $x+y$, $y+1/x$, $x^{-1}+x^{-1}y^{-1}$ and $x+x^{-1}y^{-1}$ up to units, so the closure of $C$ in the toric surface of $\Delta$ is smooth and its genus is also given by \cite{Kho78}, which is consistent with the computation above.

\subsection{The complement of the amoeba has exactly one bounded component}

For $(u,v)\in\R^2$ put $S(u,v)=e^{u}+e^{v}+e^{-u}+e^{-u-v}$. If a point $(x,y)$ has $\Log(x,y)=(u,v)$, then the four monomials of $f$ other than the constant term have total modulus $S(u,v)$, and the constant term has modulus $|\varepsilon|=10$. Therefore, if $S(u,v)<10$, then $f(x,y)\ne0$ and $(u,v)\notin\mathcal A_C$. The function $S$ is convex, being a sum of exponentials of linear functions, and it tends to infinity at infinity because the origin is an interior point of $\Delta$ and the four exponents positively span the plane; its minimum is $3.7996$ (attained at $(u,v)\approx(0.3989,-0.1995)$), which is smaller than $10$. Hence $\{S<10\}$ is a nonempty bounded convex open set contained in the complement of $\mathcal A_C$. On it the constant term dominates, so it lies in the complement component of order $(0,0)$, which is therefore bounded by the facts recalled in Section~1. The four vertices of $\Delta$ give four unbounded components of other orders. Since the order map is injective with values in the $5$ lattice points of $\Delta$ \cite{FPT}, there are no other components, and we conclude that $\R^2\setminus\mathcal A_C$ has exactly five components, among which exactly one is bounded; this is the number of interior lattice points of $\Delta$.

The computation of the amoeba confirms this. For each $u$ the points of $C$ with $\log|x|=u$ are given by the two roots $y_1,y_2$ of the quadratic above with $x=e^{u+i\varphi}$, and the section of $\mathcal A_C$ at height $u$ is the union of the two intervals described by $\min(\log|y_1|,\log|y_2|)$ and $\max(\log|y_1|,\log|y_2|)$ when $\varphi$ runs over the circle, since these two functions are continuous in $\varphi$. We drew the amoeba on a grid of $1600\times1600$ pixels in the window $[-8,8]^2$ using $6000$ values of $\varphi$ for each column, giving each section at least one pixel. Labelling the connected components of the complement of the drawn set gave five components, four of which reach the border of the window and one of which is bounded. The hole comes very close to one of the unbounded components near $(2.3,-1.2)$, but a computation with $1201\times601$ points in $[2.25,2.37]\times[-1.35,-1.05]$ shows that the amoeba keeps a positive width there, with minimum horizontal width $0.0188$ attained at $\log|y|\approx-1.159$; as the components are distinct this neck must indeed have positive width.

\subsection{The contour and its coincidence with the boundary}

For a slope $\psi\in(0,\pi)$ write $s_1=\cos\psi$ and $s_2=\sin\psi$. A point of $C$ lies over a real value of the logarithmic Gauss map exactly when $s_1\,xF_x=s_2\,yF_y$ for some such $\psi$; indeed $xF_x$ and $yF_y$ are $xy$ times $x f_x$ and $y f_y$ on $C$. A direct computation gives $s_1xF_x-s_2yF_y=y\,[A_\psi(x)\,y+B_\psi(x)]$ with $A_\psi=(s_1-2s_2)\,x$ and $B_\psi=(2s_1-s_2)x^2-(s_1-s_2)\,\varepsilon x-s_2$. Since $y\ne0$, the critical points above the slope $\psi$ are the points with $y=-B_\psi(x)/A_\psi(x)$ for which $F(x,y)=0$, and substituting we get the polynomial equation $x\,B_\psi^2-(x^2-\varepsilon x+1)\,A_\psi B_\psi+A_\psi^2=0$ in the single variable $x$. For each of $60000$ equally spaced values of $\psi$ in $(0,\pi)$ it has exactly $4$ nonzero roots, which is the degree $2\Area(\Delta)=4$ of the logarithmic Gauss map \cite{M00}. We followed the four roots from one value of $\psi$ to the next by an optimal matching of consecutive root sets, which produces four continuous curves; their images under $\Log$ form the contour. These curves are drawn in crimson in Figure~\ref{fig:amoeba} on top of the amoeba, which is shaded in a speckled blue: the shade is a gradient from pale sky blue near the boundary to deep navy deeper inside, blended with random pixel-scale noise, and it is decorative only, carrying no mathematical meaning. Figure~\ref{fig:zoom} shows an enlargement around the hole.

\begin{figure}[ht]
\centering
\includegraphics[width=0.26\textwidth]{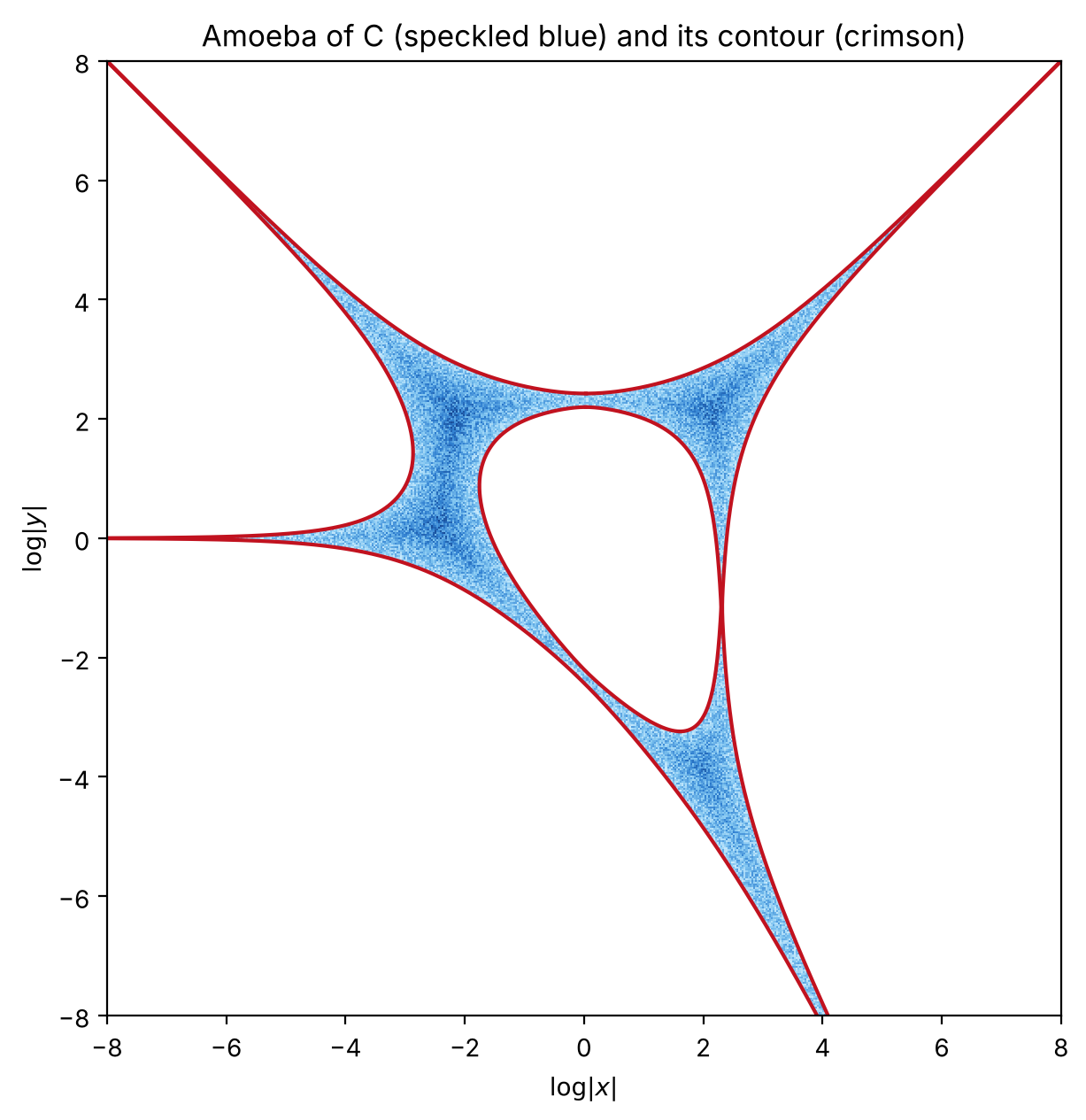} \qquad   \includegraphics[width=0.26\textwidth]{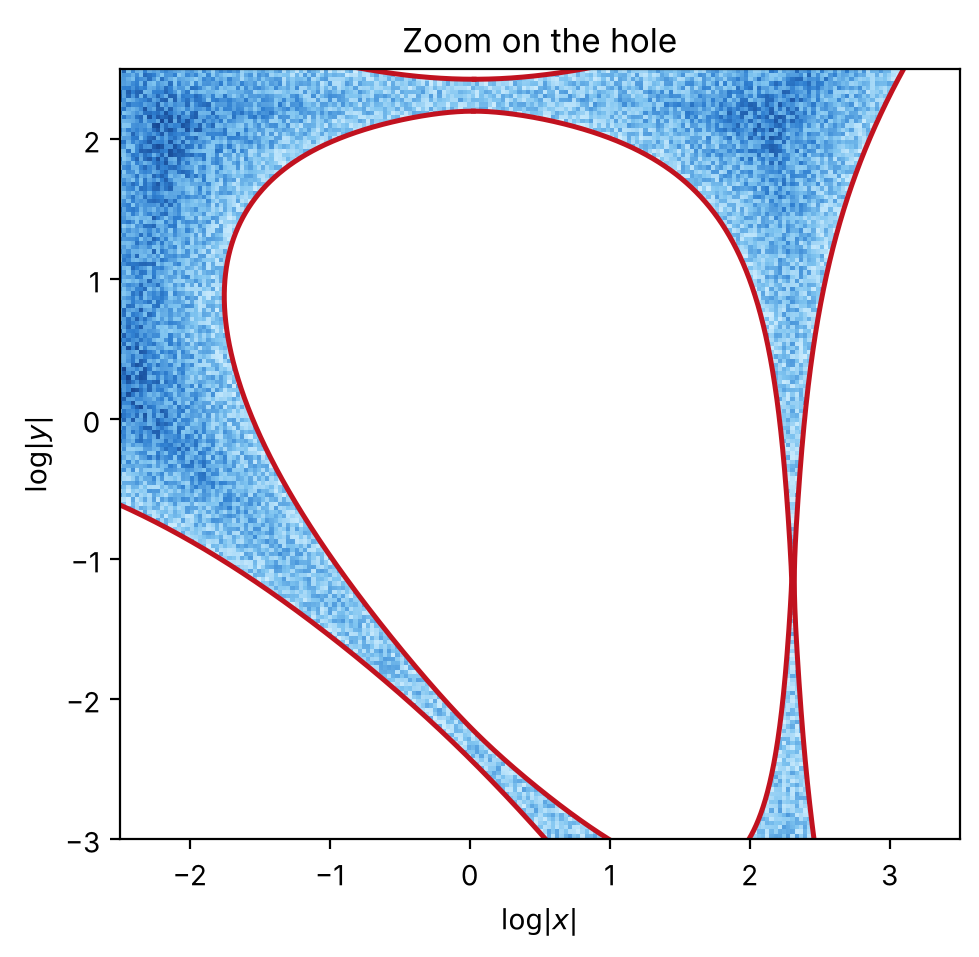}
\caption{The amoeba of $C$ (speckled blue) with its contour (crimson) in the window $[-8,8]^2$. There are four tentacles, in the directions $(1,1)$, $(-1,1)$, $(-1,0)$ and $(1,-2)$ of the outer normals of $\Delta$, and one hole. On the right, enlargement around the hole; the thin neck of the amoeba near $(2.3,-1.2)$ has width about $0.019$.}
\label{fig:amoeba}
\end{figure}

We then tested numerically that the contour has no part inside the amoeba. For each of the $238137$ points of the four crimson curves lying in the window $[-7.8,7.8]^2$ we computed its distance to the complement of the amoeba using the pixel image; the largest value is $0.02$, that is two pixels, so that every contour point lies on the boundary at the resolution of the picture. Conversely every boundary pixel is within $0.0142$ of a contour point. The inclusion of the boundary in the contour is a general fact, so the two sets agree in the window up to the resolution. The contour is smooth in the sense that the four curves have no cusp: the largest turning angle between consecutive segments of the polygonal approximations inside the window is $0.125$ radians. Where the curves leave the window the successive tangent directions reverse, with turning angle $\pi$, at the three slopes $\psi=\pi/4$, $\psi=\arctan2$ and $\psi=3\pi/4$, at points of $\Log$-norm larger than $10$; at these slopes a root goes to $0$ or $\infty$ along a tentacle, with $\Log$-image running out along one side of the tentacle and coming back along the other side, so that no cusp is involved. Finally, the curve deforms continuously to the real curve $x+y+x^{-1}+x^{-1}y^{-1}-10$ by rotating the phase of the constant term. For that real curve the area of the amoeba computed by the same method is $19.741$, equal to $2\pi^2\approx19.739$ within the error, so it behaves like a Harnack curve for which the contour is the real part and equals the boundary \cite{Mikhalkin2000}; for phases $0.3$, $0.6$, $0.9$ and $1.2$ of the constant term $-10\,e^{i\theta}$ the areas are $19.494$, $18.926$, $18.166$ and $17.537$, decreasing as the curve moves away from the real one. This suggests that the coincidence of contour and boundary persists under the deformation, but we stress that our verification of it for the curve $C$ is a numerical one, not a proof.

\subsection{Why $C$ is not toric-equivalent to a simple Harnack curve}

The exponents $(1,0)$, $(-1,0)$ and $(0,0)$ satisfy the affine relation $(1,0)+(-1,0)-2\,(0,0)=0$, and the exponents $(0,1)$, $(-1,-1)$, $(-1,0)$, $(0,0)$ satisfy $(0,1)+(-1,-1)-(-1,0)-(0,0)=0$. Let $c_\alpha$ denote the coefficient of $x^{\alpha_1}y^{\alpha_2}$ in $f$, so that $c_{(1,0)}=c_{(0,1)}=c_{(-1,0)}=c_{(-1,-1)}=1$ and $c_{(0,0)}=-\varepsilon$. The first quantity is $I_1=c_{(1,0)}c_{(-1,0)}/c_{(0,0)}^{2}=1/\varepsilon^{2}=-0.0028-0.0096\,i$ and the second is $I_2=c_{(0,1)}c_{(-1,-1)}/(c_{(-1,0)}c_{(0,0)})=-1/\varepsilon=-0.06+0.08\,i$. Both do not change when $f$ is multiplied by a constant $\kappa$ or by a monomial, nor when $x$ and $y$ are replaced by $\lambda x$ and $\mu y$: the scalar $\kappa$ occurs to the same total power in the numerator and in the denominator, and the factors $\lambda^{\alpha_1}\mu^{\alpha_2}$ cancel exactly because the corresponding exponents satisfy the affine relations above. A change of coordinates with matrix in $GL_2(\Z)$ transports the support of $f$ onto the support of the new polynomial and the affine relations onto affine relations, so such products of coefficients computed for the transported relations are unchanged as well.

Suppose now that $\Phi$ is a monomial change of coordinates carrying $C$ onto a simple Harnack curve $C'$. By definition \cite{Mikhalkin2000} the curve $C'$ is real, that is, defined by a polynomial $g$ with real coefficients. The Laurent polynomial $g\circ\Phi$ vanishes exactly on $C$, which is irreducible by Section~3, so that $g\circ\Phi$ is $\kappa$ times a monomial times $f$; in particular $g$ has five terms, its support corresponds to the support of $f$, and the corresponding quantities $I_1$ and $I_2$ computed from the coefficients of $g$ are real and equal to the ones computed from $f$. But $I_2=-0.06+0.08\,i$ is not real, a contradiction. Hence $C$ is not toric-equivalent to a simple Harnack curve, and the argument covers equivalences with arbitrary complex translations of the torus.

The area gives an independent confirmation. The area of $\mathcal A_C$, computed by integrating over $u\in[-40,40]$ the total length of the section described in Section~4, is $18.110$, and it is the same, up to $10^{-4}$, for the window $[-25,25]$. The maximal value for this polygon is $\pi^2\Area(\Delta)=2\pi^2\approx19.739$ by \cite{PassareRullgard}, so $\mathcal A_C$ fills about $91.7\%$ of it. By \cite{MR2001} a curve with this Newton polygon and this amoeba cannot be Harnack, and since both areas are preserved by toric equivalence it cannot be toric-equivalent to a Harnack curve either.


\section{Example of genus three curve}

In this section,
we give an explicit irreducible curve $C\subset(\C^*)^2$ defined by a polynomial with seven terms and Gaussian integer coefficients. Its Newton polygon is the quadrilateral with vertices $(0,1)$, $(1,0)$, $(4,1)$, $(2,2)$, which has no pair of parallel sides, so it is neither a triangle nor a zonotope, and it has exactly three interior lattice points, $(1,1)$, $(2,1)$ and $(3,1)$. The curve is smooth in the torus and has genus $3$. The complement of its amoeba has exactly seven components, four unbounded ones and three bounded ones, and the three bounded components have orders $(1,1)$, $(2,1)$ and $(3,1)$, so their number equals the number of interior lattice points. The critical contour of the amoeba map is a smooth curve, without cusps, which in our computation coincides with the topological boundary of the amoeba. The curve is nevertheless not toric-equivalent to a simple Harnack curve, because two torus-invariant combinations of its coefficients are not real, and because the area of its amoeba, about $32.35$, is smaller than $\pi^2$ times the area $4$ of the Newton polygon, which is about $39.48$. We display the amoeba in a speckled blue, its contour in crimson, both together, and the Newton polygon, and we say precisely which statements are proved from the cited theorems and exact computations and which are numerical.
 

\begin{theorem}\label{thm:main}
Let $P(x)=\tfrac14+16i\,x+(11-63i)\,x^2+16\,x^3-\tfrac14\,x^4$, let $G(x,y)=x+x^2y^2+y\,P(x)$ and let $C=\{G=0\}\subset(\C^*)^2$. Then the following hold. 
\begin{itemize}
\item[(i)]  The polynomial $G$ is irreducible, $C$ is smooth, and $C$ has genus $3$;
\item[(ii)]   the Newton polygon $\Delta$ of $G$ is the quadrilateral with vertices $(0,1)$, $(1,0)$, $(4,1)$, $(2,2)$; it has no pair of parallel sides, so it is neither a triangle nor a zonotope, and its interior lattice points are exactly $(1,1)$, $(2,1)$ and $(3,1)$; 
\item[(iii)]   the complement of the amoeba $\mathcal A_C$ has exactly seven connected components, four unbounded and three bounded, and the three bounded ones have orders $(1,1)$, $(2,1)$ and $(3,1)$; in particular their number equals the number of interior lattice points of $\Delta$; 
\item[(iv)]   the curve $C$ is not toric-equivalent to any simple Harnack curve, and $\Area(\mathcal A_C)<\pi^2\Area(\Delta)=4\pi^2$;
\item[(V)]   the critical contour of $\Log|_C$ is a smooth curve without cusps which coincides with the topological boundary of $\mathcal A_C$.
\end{itemize}
\end{theorem}

\begin{proof}
The first and second assertions are proved in Sub-sections~\ref{sec:curve} and~\ref{sec:genus} (the Newton polygon by Pick's formula, smoothness and irreducibility from the discriminant, whose eight roots are simple by an exact resultant computation, and the genus by Riemann--Hurwitz and Khovanskii's theorem \cite{Kho78}). The third assertion is proved in Sun-section~\ref{sec:holes} from the dominance of a single monomial at three explicit points and from the injectivity of the order map \cite{FPT,PassareRullgard}. The first part of the fourth assertion is proved in Section~\ref{sec:harnack} with two non-real torus-invariant products of coefficients, and the area inequality is established there by an integration of the sections of the amoeba, whose numerical value is about $32.35$. The fifth assertion is established in Section~\ref{sec:contour} by a numerical computation of the contour and of the boundary at a resolution of $0.01$.
\end{proof}

\begin{remark}
The first four assertions follow from the cited theorems and from exact or certified computations, with the single exception of the numerical value of the area, whose inequality has a wide margin ($32.35$ against $39.48$). The fifth assertion is the only one that rests on a numerical experiment alone; it is stated as part of the theorem because it is what the computation of Section~\ref{sec:contour} shows, and Section~\ref{sec:contour} also shows that it depends on the choice of the coefficients.
\end{remark}

\subsection{The curve and its Newton polygon}\label{sec:curve}

Put $P(x)=\tfrac14+16i\,x+(11-63i)\,x^2+16\,x^3-\tfrac14\,x^4$ and let $C\subset(\C^*)^2$ be the zero set of
$G(x,y)=x+x^2y^2+y\,P(x).$
Multiplying by $4$ we obtain the polynomial with Gaussian integer coefficients
$F(x,y)=4x+4x^2y^2+\bigl(1+64i\,x+(44-252i)\,x^2+64\,x^3-x^4\bigr)\,y,$
which has the same zero set. The support of $F$ consists of the seven exponents $(1,0)$, $(2,2)$, $(0,1)$, $(1,1)$, $(2,1)$, $(3,1)$ and $(4,1)$, and the Newton polygon $\Delta$ is the convex hull of $(0,1)$, $(1,0)$, $(4,1)$ and $(2,2)$. Going around its boundary counterclockwise, the sides are the vectors $(1,-1)$, $(3,1)$, $(-2,1)$ and $(-2,-1)$. Each of them is primitive, their sum is zero, and the cross products of consecutive sides are $4$, $5$, $4$ and $3$, all positive, so the four points are the vertices of a convex quadrilateral. This quadrilateral is not a triangle, and since no two of its sides are parallel it is not a trapezoid, a fortiori not a parallelogram and not centrally symmetric, hence not a zonotope. The shoelace formula gives $\Area(\Delta)=4$, the boundary contains exactly $4$ lattice points because every side is primitive, and Pick's formula gives $4-4/2+1=3$ interior lattice points, namely $(1,1)$, $(2,1)$ and $(3,1)$, which all occur in $F$. The polygon is drawn in Figure~\ref{fig:newton}; the vertices are the red dots and the three interior points are the dark blue dots.

\begin{figure}[ht]
\centering
\includegraphics[width=0.29\textwidth]{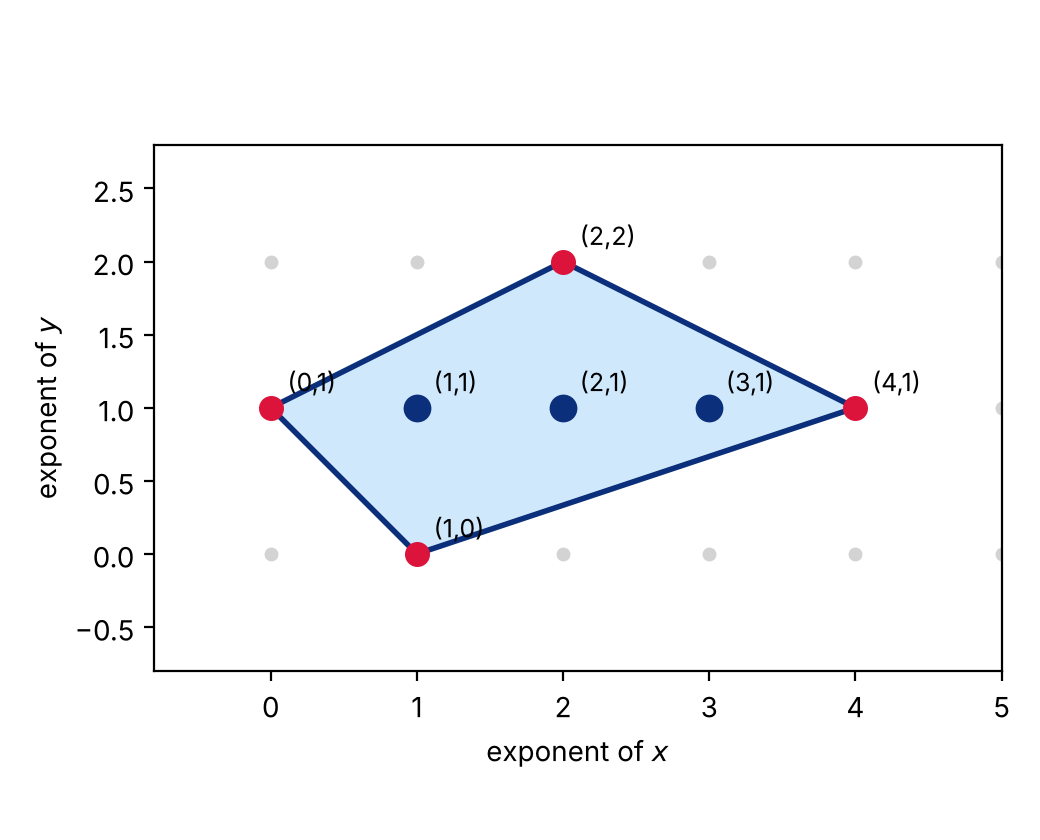}
\caption{The Newton polygon $\Delta$ of $F$: a quadrilateral with vertices $(0,1)$, $(1,0)$, $(4,1)$, $(2,2)$ and the three interior lattice points $(1,1)$, $(2,1)$, $(3,1)$ (dark blue). Grey dots are the other points of the lattice.}
\label{fig:newton}
\end{figure}

\subsection{Smoothness, irreducibility and genus}\label{sec:genus}

Regard $G=x^2\,y^2+P(x)\,y+x$ as a quadratic polynomial in $y$ over $\C(x)$. Its discriminant is $D(x)=P(x)^2-4x^3$, a polynomial of degree $8$ with leading coefficient $1/16$ and $D(0)=1/16$. The eight roots of $D$ are simple. We checked this exactly: the coefficients of $D$ and of $D'$ are Gaussian rationals, and the resultant of $D$ and $D'$, computed by Gaussian elimination with exact rational arithmetic on the Sylvester matrix of size $15$, is $-7.3232\cdot10^{15}-1.4044\cdot10^{15}\,i$ up to rounding, hence different from $0$. The numerical values of the roots, together with the point $(x_0,y_0)$ of $C$ above each of them, are given in Table~\ref{tab:roots}; they come in four nearby pairs, which reflects the strong separation of the sizes of the coefficients, and the smallest distance inside a pair is $4.6\cdot10^{-4}$.

\begin{table}[ht]
\centering
\small
\renewcommand{\arraystretch}{1.25}
\begin{tabular}{@{}lll@{}}
\hline
root $x_0$ of $D$ & $y_0=-P(x_0)/(2x_0^2)$ & $|\partial_xG(x_0,y_0)|$ \\
\hline
$64.9161-3.8125\,i$ & $0.12395+0.00364\,i$ & $8689$ \\
$64.8865-3.8169\,i$ & $-0.12398-0.00364\,i$ & $8673$ \\
$-1.1757+3.9132\,i$ & $0.29523-0.39696\,i$ & $142.6$ \\
$-1.1001+3.8249\,i$ & $-0.30150+0.40044\,i$ & $135.0$ \\
$0.24265-0.05625\,i$ & $-1.99069-0.22771\,i$ & $33.97$ \\
$0.23240-0.08318\,i$ & $1.98313+0.34422\,i$ & $34.80$ \\
$-0.00077+0.01555\,i$ & $-5.52502+5.80576\,i$ & $130.3$ \\
$-0.00102+0.01516\,i$ & $5.53888-5.92574\,i$ & $136.8$ \\
\hline
\end{tabular}
\caption{The eight branch points of the projection $(x,y)\mapsto x$ on $C$, with the value of $\partial_xG$ at the corresponding point of $C$.}
\label{tab:roots}
\end{table}

Over a value $x_0\in\C^*$ with $D(x_0)\ne0$ the curve has two distinct points, and at those points $\partial_yG\neq0$, so $C$ is smooth there. Over a root $x_0$ of $D$ there is one point $(x_0,y_0)$, with $y_0=-P(x_0)/(2x_0^2)$, at which $\partial_yG=0$. If $\partial_xG$ vanished there as well, $x_0$ would be a multiple root of the discriminant of a quadratic in $y$, which is impossible for a simple root; Table~\ref{tab:roots} shows directly that $|\partial_xG|$ is between $33$ and $8700$ at the eight points. Hence $C$ is smooth in the torus.

If $G$ were reducible, the quadratic in $y$ would have a root in $\C(x)$, because its coefficients $x^2$, $P(x)$ and $x$ have no common factor, and then $D$ would be a square in $\C(x)$; this is false since $D$ has simple roots. Hence $G$, and therefore $F$, is irreducible.

To compute the genus, put $w=2x^2y+P(x)$. On $C$ we have $4x^2G=0$, that is $w^2=P^2-4x^3=D(x)$, and the map $(x,y)\mapsto(x,w)$ identifies $C$ with an open subset of the affine curve $w^2=D(x)$. Its smooth compactification is a hyperelliptic curve, which by Riemann--Hurwitz (see \cite{Har77}, Chapter~IV, Proposition~2.4) has genus $g$ with $2g-2=2\cdot(-2)+8$, that is $g=3$, because $D$ has degree $8$ and eight simple roots. It remains to see that the compactification of $C$ does not acquire further branching over $x=0$ and $x=\infty$. Over $x=0$ the two roots of the quadratic satisfy $y\sim-x/P(0)=-4x\to0$ and $y\sim-P(0)/x^2=-1/(4x^2)\to\infty$, and over $x=\infty$ they satisfy $y\sim x^2/4$ and $y\sim4/x^3$; in both cases they are distinct points of $\mathbb{P}^1$, so the double cover is unramified there. Consequently $C$ has geometric genus $3$, which is the number of interior lattice points of $\Delta$. Each side of $\Delta$ is primitive, so it carries only its two end points, and the polynomial formed by the two corresponding monomials of $G$ is a binomial, which has a single simple root in $\C^*$; so the closure of $C$ in the toric surface of $\Delta$ is smooth and its genus is also given by \cite{Kho78}, in agreement with the computation above.

\subsection{The complement of the amoeba has exactly three bounded components}\label{sec:holes}

Let $A=(-2.78,\,0.85)$, $B=(0,\,0.5)$ and $C_0=(2.8,\,-2.0)$ be three points of the $(u,v)$-plane. For each of them we computed the moduli of the seven terms of $G$ on the torus $\{\Log(x,y)=(u,v)\}$; the moduli of $x$, $x^2y^2$, $p_0y$, $p_1xy$, $p_2x^2y$, $p_3x^3y$ and $p_4x^4y$ are $e^u$, $e^{2u+2v}$, $|p_k|e^{ku+v}$ for $k=0,\dots,4$, with $|p_0|=|p_4|=1/4$, $|p_1|=|p_3|=16$ and $|p_2|=\sqrt{11^2+63^2}=63.96$. They are collected in Table~\ref{tab:dom}.

\begin{table}[ht]
\centering
\scriptsize
\setlength{\tabcolsep}{3pt}
\renewcommand{\arraystretch}{1.25}
\begin{tabular}{@{}lccccccc@{}}
\hline
point & $x$ & $x^2y^2$ & $p_0y$ & $p_1xy$ & $p_2x^2y$ & $p_3x^3y$ & $p_4x^4y$ \\
\hline
$A=(-2.78,0.85)$ & $0.0620$ & $0.0211$ & $0.5849$ & $\mathbf{2.3224}$ & $0.5759$ & $0.0089$ & $9\cdot10^{-6}$ \\
$B=(0,0.5)$ & $1$ & $2.718$ & $0.4122$ & $26.38$ & $\mathbf{105.44}$ & $26.38$ & $0.4122$ \\
$C_0=(2.8,-2.0)$ & $16.44$ & $4.953$ & $0.0338$ & $35.61$ & $2341$ & $\mathbf{9630}$ & $2474$ \\
\hline
\end{tabular}
\caption{Moduli of the seven terms of $G$ at three points of the complement; the dominant term is in bold. The sum of the other six moduli is $1.2528$, $57.302$ and $4871.9$ respectively, that is $54\%$, $54\%$ and $51\%$ of the dominant one.}
\label{tab:dom}
\end{table}

At each of these points one term has modulus larger than the sum of the other six. Therefore $G$ has no zero on the corresponding torus, the three points belong to the complement of $\mathcal A_C$, and by Section~1 they lie in components whose orders are the exponents of the dominant terms, namely $(1,1)$, $(2,1)$ and $(3,1)$. We also computed the average of $(xG_x/G,\,yG_y/G)$ over the torus with a grid of $600\times600$ points, a computation that converges exponentially because the integrand is analytic and $|G|$ is bounded away from $0$ (the minimum of $|G|$ on a grid of $2000\times2000$ points is $1.747$, $83.84$ and $7499$, and the Lipschitz constants of $G$ with respect to the two arguments, multiplied by half the grid spacing $0.0031$, are smaller than $0.012$, $0.77$ and $92$ respectively, so $|G|\geq1.735$, $83.07$ and $7408$ on the whole tori); the results are $(1,1)$, $(2,1)$ and $(3,1)$ up to $10^{-6}$. These orders are interior lattice points, so the three components are bounded, and they are different because their orders are different. The four vertices of $\Delta$ are the orders of four unbounded components. Since the order map is injective with values in the $7$ lattice points of $\Delta$ \cite{FPT}, there are no other components, and we conclude that $\R^2\setminus\mathcal A_C$ has exactly seven components, among which exactly three are bounded; this is the number of interior lattice points of $\Delta$.

The computation of the amoeba confirms this and shows the shape of the components. For each $u$ the points of $C$ with $\log|x|=u$ are given by the two roots $y_1,y_2$ of the quadratic above with $x=e^{u+i\varphi}$, and the section of $\mathcal A_C$ at abscissa $u$ is the union of the two intervals described by $\min(\log|y_1|,\log|y_2|)$ and $\max(\log|y_1|,\log|y_2|)$ when $\varphi$ runs over the circle, since these two functions are continuous in $\varphi$. Because the amoeba has long almost vertical tentacles and strips, we also computed the sections at fixed ordinate $v$: for $y=e^{v+i\varphi}$ the points of $C$ are given by the four roots in $x$ of the quartic $p_4y\,x^4+p_3y\,x^3+(p_2y+y^2)\,x^2+(p_1y+1)\,x+p_0y$, and the four roots ordered by modulus are continuous functions of $\varphi$, so each of the four intervals described by the $k$-th smallest value of $\log|x|$ lies in the section; we took the union of all these sections. The grid has $2400\times2400$ pixels in the window $[-12,12]^2$, with $5000$ values of $\varphi$ for the sections in $u$ and $3000$ for those in $v$, and every interval was given at least one pixel. Labelling the connected components of the complement of the drawn set gave seven components, four of which reach the border of the window and three of which are bounded. The three bounded components are contained in the boxes $[-4.15,-1.42]\times[-2.52,5.60]$ (order $(1,1)$), $[-1.38,1.37]\times[-4.25,3.96]$ (order $(2,1)$) and $[1.43,4.16]\times[-9.24,5.60]$ (order $(3,1)$), and their areas on the grid are $17.9$, $18.1$ and $33.2$.

The holes are separated from each other and from the unbounded components by four almost vertical strips of the amoeba, located near $u=-4.17$, $u=-1.39$, $u=1.39$ and $u=4.17$, and these strips are so thin in places that in a figure two neighbouring complement components seem to touch. We therefore measured their width directly with the sections at fixed $v$: for each $v$ the section consists of four disjoint intervals of $u$, one for each strip, and the width of the strip is the length of the corresponding interval. Figure~\ref{fig:neck} gives these widths as functions of $v$ on a logarithmic scale and Table~\ref{tab:neck} their minima. The minima are $0.024$, $0.0091$, $0.0263$ and $4.5\cdot10^{-4}$; they were computed with $40000$ values of $\varphi$ and confirmed with $6000$ values and with a second root finder, which gives the same four numbers to three digits. They are all positive, so the five components that the strips separate are indeed pairwise disjoint, in agreement with the count of orders above; but the smallest one, which separates the hole of order $(3,1)$ from the unbounded component on the right, is smaller than the pixel of any figure, and the reader should keep in mind that a figure can only show it as a line.

\begin{table}[ht]
\centering
\small
\renewcommand{\arraystretch}{1.25}
\begin{tabular}{@{}lp{6.2cm}cc@{}}
\hline
strip, centred at $u\approx$ & it separates & $v$ of the minimum & minimal width \\
\hline
$-4.1745$ & unbounded component on the left from the hole $(1,1)$ & $2.10$ & $0.0240$ \\
$-1.3945$ & hole $(1,1)$ from hole $(2,1)$ & $0.70$ & $0.0091$ \\
$1.3945$ & hole $(2,1)$ from hole $(3,1)$ & $-0.70$ & $0.0263$ \\
$4.1746$ & hole $(3,1)$ from the unbounded component on the right & $-2.10$ & $4.5\cdot10^{-4}$ \\
\hline
\end{tabular}
\caption{The four strips of the amoeba between the complement components, with the minimal width of each of them in the direction of the $u$-axis.}
\label{tab:neck}
\end{table}

\begin{figure}[ht]
\centering
\includegraphics[width=0.45\textwidth]{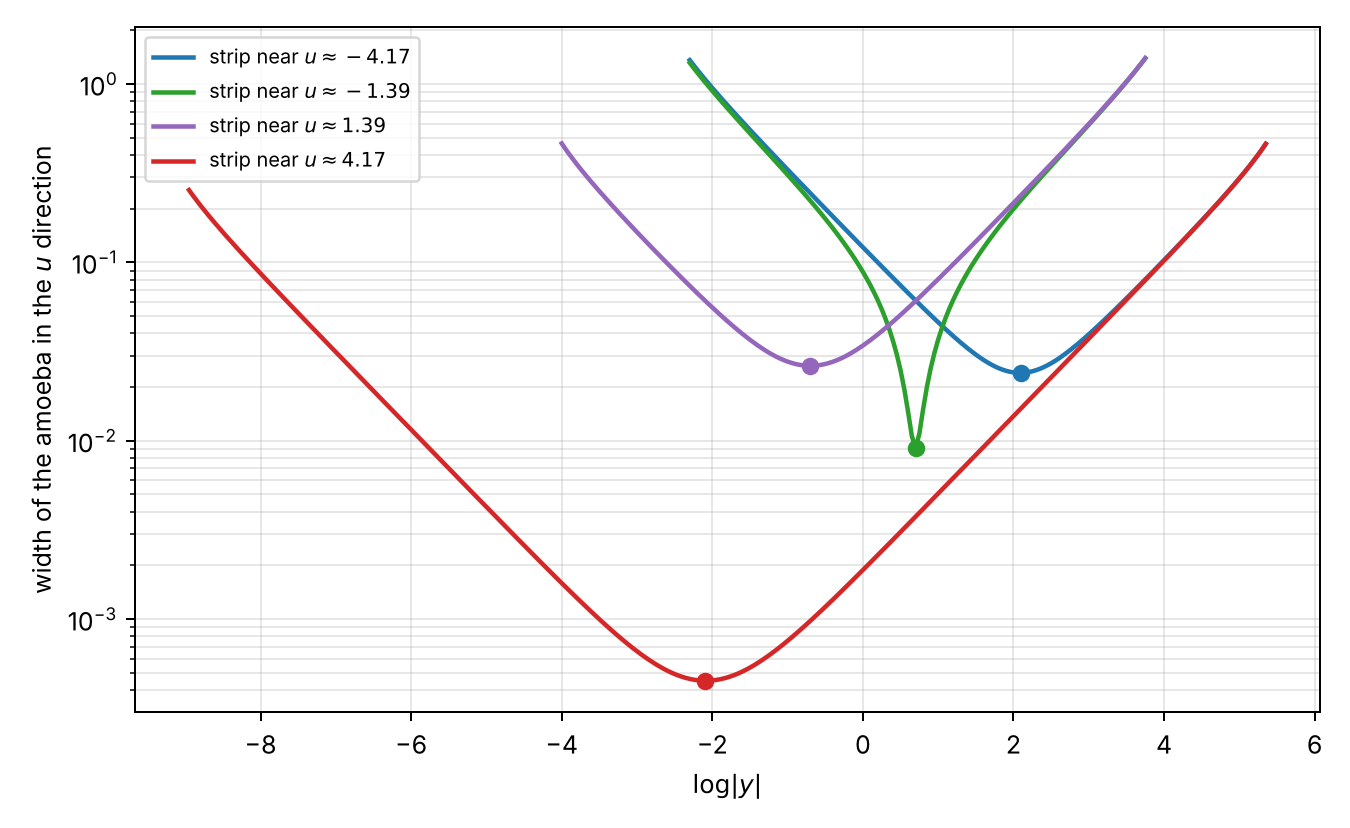}
\caption{Width of each of the four strips of the amoeba, in the direction of the $u$-axis, as a function of $v=\log|y|$ (logarithmic scale); each curve is drawn over the range of $v$ in which the strip lies between the two components that it separates, and the dots mark the minima of Table~\ref{tab:neck}.}
\label{fig:neck}
\end{figure}

\subsection{The contour and its coincidence with the boundary}\label{sec:contour}

For a slope $\psi\in(0,\pi)$ write $s_1=\cos\psi$ and $s_2=\sin\psi$. A point of $C$ lies over a real value of the logarithmic Gauss map exactly when $s_1\,xG_x=s_2\,yG_y$ for some such $\psi$. A direct computation gives
$s_1xG_x-s_2yG_y=2(s_1-s_2)\,x^2y^2+(s_1\,xP'(x)-s_2P(x))\,y+s_1x,$
and subtracting $2(s_1-s_2)\,G$, which vanishes on $C$, we obtain the expression $a_\psi(x)+b_\psi(x)\,y$ with
$a_\psi(x)=(2s_2-s_1)\,x,\qquad b_\psi(x)=s_1\,xP'(x)+(s_2-2s_1)\,P(x).$
Hence the critical points above the slope $\psi$ are the points of $C$ with $b_\psi(x)\,y=-a_\psi(x)$. When $b_\psi(x)\ne0$ we have $y=-a_\psi(x)/b_\psi(x)$, and substituting in $G=0$ and multiplying by $b_\psi^2$ we get $x\,b_\psi^2-P\,a_\psi b_\psi+x^2a_\psi^2=0$, which is $x$ times the polynomial
$R_\psi(x)=b_\psi(x)^2-(2s_2-s_1)\,P(x)\,b_\psi(x)+(2s_2-s_1)^2\,x^3$
of degree $8$ in the single variable $x$. Its leading coefficient is $p_4^2\,(2s_1+s_2)(3s_1-s_2)$ and its constant term is $p_0^2\,(2s_1-s_2)(s_1+s_2)$, so that for a slope outside the four values $\arctan3\approx1.2490$, $\pi-\arctan2\approx2.0344$, $\arctan2\approx1.1071$ and $3\pi/4\approx2.3562$ the polynomial $R_\psi$ has eight nonzero roots, which is the degree $2\Area(\Delta)=8$ of the logarithmic Gauss map \cite{Mikhalkin2000}. The four exceptional slopes are the ones for which a root goes to $0$ or $\infty$ along a tentacle: they are the directions of the sides $(3,1)$, $(-2,1)$, $(-2,-1)$ and $(1,-1)$, in the sense that the tangent $(\cos\psi,-\sin\psi)$ of the contour is the direction of the corresponding tentacle. A fifth slope, $\psi_0=\arctan(1/2)\approx0.4636$, is special for another reason: there $a_\psi\equiv0$, the polynomial $R_{\psi_0}=b_{\psi_0}^2$ has four double roots, and each root $x_0$ of $b_{\psi_0}$ gives two critical points, with the two values of $y$ for which $G(x_0,y)=0$.

We computed these critical points for $120000$ equally spaced values of $\psi$ in $[10^{-4},\pi-10^{-4}]$: we took the eight roots of $R_\psi$ with a numerical root finder, set $y=-a_\psi/b_\psi$, and refined each pair $(\log x,\log y)$ with eight iterations of Newton's method on the two equations $G=0$ and $s_1xG_x=s_2yG_y$ written in logarithmic coordinates, which also repairs the loss of accuracy near $\psi_0$ and near the exceptional slopes. The eight points were followed from one value of $\psi$ to the next by an optimal matching between the predicted and the computed points, with distances measured in the complex logarithmic coordinates; the prediction by extrapolation is necessary because near $\psi=\pi/2$ the two arcs that run along the thinnest strip move in opposite directions with velocity about $4200$ in $v$ per unit of $\psi$ while being only $4.5\cdot10^{-4}$ apart in $u$. This produces eight continuous curves, whose images under $\Log$ form the contour. They are drawn in crimson in Figure~\ref{fig:contour}, alone, and in Figure~\ref{fig:combined} on top of the amoeba.  

\begin{figure}[ht]
\centering
\includegraphics[width=0.23\textwidth]{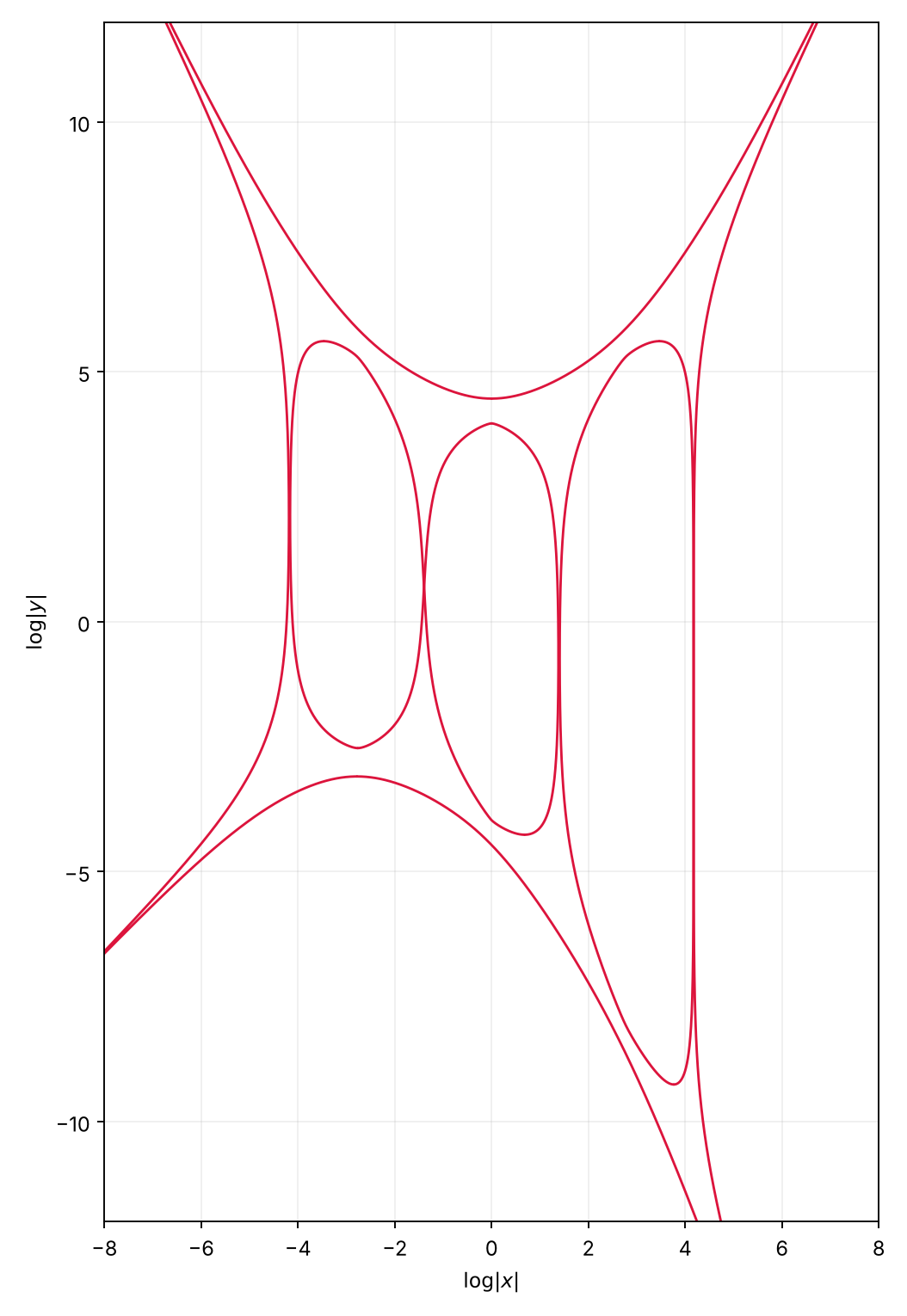}\quad   \includegraphics[width=0.23\textwidth]{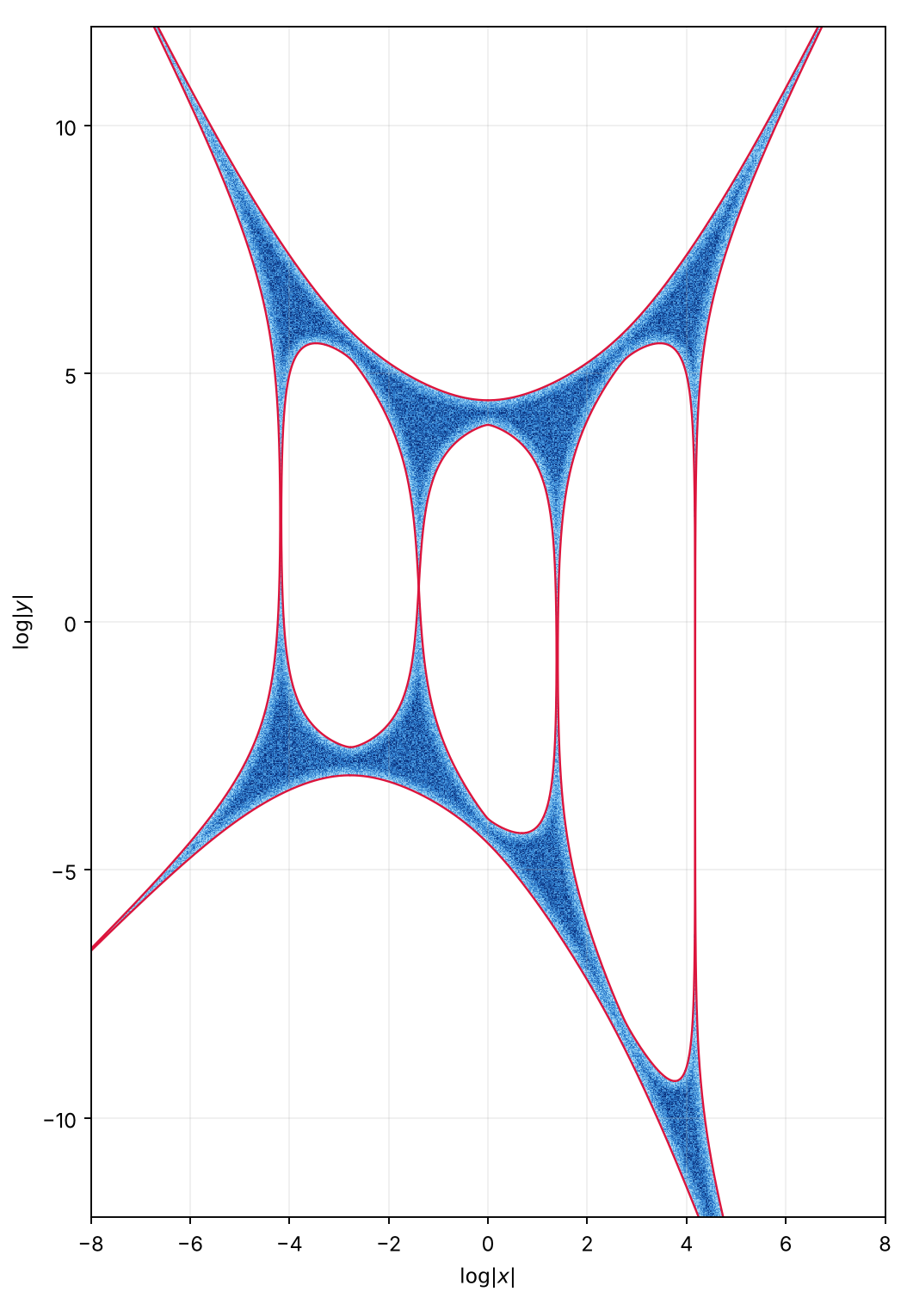} \quad \includegraphics[width=0.35\textwidth]{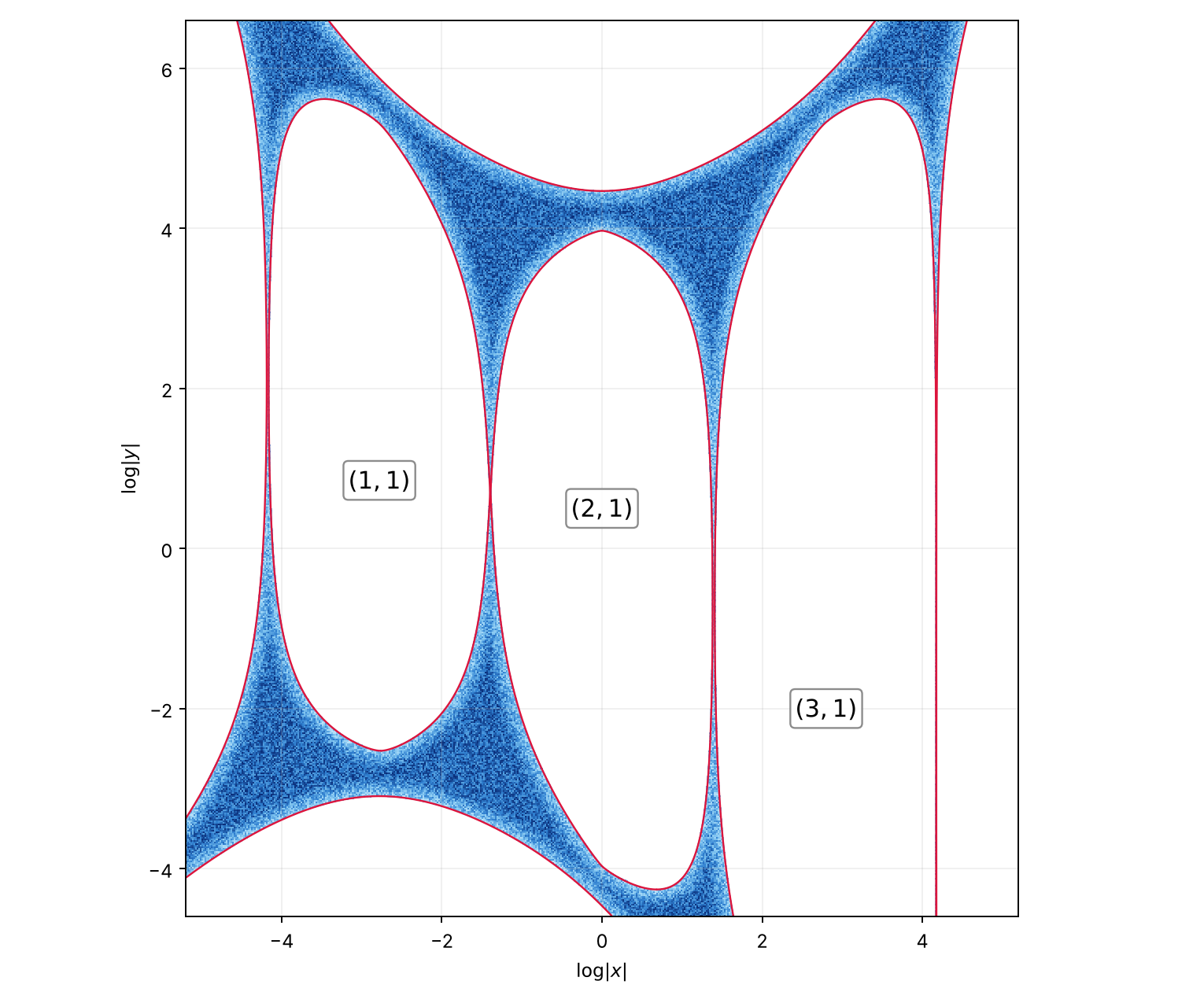}
\caption{The contour of the amoeba of $C$ in the window $[-8,8]\times[-12,12]$, drawn in crimson alone. It consists of eight smooth arcs; the boundary of each hole is a closed curve made of arcs of the contour, and the outer arcs run off along the four tentacles. In the middle, The amoeba of $C$ (speckled blue) with its contour (crimson) in the window $[-8,8]\times[-12,12]$. In the left,  Enlargement around the three holes, in the window $[-5.2,5.2]\times[-4.6,6.6]$; the number in each hole is its order.}
\label{fig:contour}
\end{figure}

We then tested numerically that the contour has no part inside the amoeba and is smooth. For each of the $950835$ points of the eight curves lying in the window $[-11.7,11.7]^2$ we computed its distance to the complement of the amoeba using the pixel image of Section~4; the largest value is $0.02$, that is two pixels, so every contour point lies on the boundary at the resolution of the picture. Conversely, rasterizing the polygonal lines through consecutive points of the eight curves, every boundary pixel of the amoeba in the window is within $0.01$, that is one pixel, of the contour. The inclusion of the boundary in the contour is a general fact, so the two sets agree in the window up to the resolution. The eight curves have no cusp: the largest turning angle between consecutive segments of the polygonal approximations inside the window is $2.6\cdot10^{-5}$ radians, and the speed $|d\,\Log/d\psi|$ stays above $0.089$ everywhere in the window. Moreover the eight critical points above a slope $\psi$ are pairwise distinct points of $C$ at every tenth of the slopes we sampled, with relative distance at least $0.62$, so no real value of $\gamma$ is a critical value and $\gamma^{-1}(\mathbb{RP}^1)$ is a smooth curve; this excludes the junction points that occur in other examples at real ramification points of $\gamma$. Where the arcs leave the window they run along the two sides of each of the four tentacles, so that each tentacle is bounded by two arcs and there are eight ends in total; at the four exceptional slopes a root goes to $0$ or $\infty$ and the arcs run out to infinity, with no cusp involved.

The coincidence of contour and boundary is not automatic for a curve with a given Newton polygon and three holes. With essentially the same phases of the coefficients but with $|p_0|=|p_4|=1$ instead of $1/4$ (and $|p_1|=|p_3|=16$, $|p_2|=64$), the same method gives a smooth contour that has pieces at distance up to $0.37$ inside the amoeba, so the property is tied to the strongly unbalanced size of the coefficients that we chose; for a simple Harnack curve it holds for all coefficients. 

\subsection{Why $C$ is not toric-equivalent to a simple Harnack curve}\label{sec:harnack}

Let $c_\alpha$ denote the coefficient of $x^{\alpha_1}y^{\alpha_2}$ in $G$, so that $c_{(1,0)}=c_{(2,2)}=1$, $c_{(0,1)}=1/4$, $c_{(1,1)}=16i$, $c_{(2,1)}=11-63i$, $c_{(3,1)}=16$ and $c_{(4,1)}=-1/4$. The exponents $(1,1)$, $(2,1)$ and $(3,1)$ satisfy the affine relation $(1,1)+(3,1)-2\,(2,1)=0$, and the exponents $(1,0)$, $(2,2)$, $(1,1)$ and $(2,1)$ satisfy $(1,0)+(2,2)-(1,1)-(2,1)=0$. The corresponding products of coefficients are
$I_1=c_{(1,1)}c_{(3,1)}/c_{(2,1)}^{2}=256\,i/(11-63i)^2=-0.02121-0.05889\,i$ and $I_2=c_{(1,0)}c_{(2,2)}/(c_{(1,1)}c_{(2,1)})=1/(16i\,(11-63i))=0.00096-0.00017\,i$.
Both do not change when $G$ is multiplied by a constant $\kappa$ or by a monomial, nor when $x$ and $y$ are replaced by $\lambda x$ and $\mu y$: the scalar $\kappa$ occurs to the same total power in the numerator and in the denominator, and the factors $\lambda^{\alpha_1}\mu^{\alpha_2}$ cancel exactly because the corresponding exponents satisfy the affine relations above. A change of coordinates with matrix in $GL_2(\Z)$ transports the support of $G$ onto the support of the new polynomial and the affine relations onto affine relations, so such products of coefficients computed for the transported relations are unchanged as well.

Suppose now that $\Phi$ is a monomial change of coordinates carrying $C$ onto a simple Harnack curve $C'$. By definition \cite{Mikhalkin2000} the curve $C'$ is real, that is, defined by a polynomial $g$ with real coefficients. The Laurent polynomial $g\circ\Phi$ vanishes exactly on $C$, which is irreducible by Section~3, so that $g\circ\Phi$ is $\kappa$ times a monomial times $G$; in particular $g$ has seven terms, its support corresponds to the support of $G$, and the corresponding quantities $I_1$ and $I_2$ computed from the coefficients of $g$ are real and equal to the ones computed from $G$. But $I_1=-0.02121-0.05889\,i$ is not real, a contradiction. Hence $C$ is not toric-equivalent to a simple Harnack curve, and the argument covers equivalences with arbitrary complex translations of the torus.

The area gives an independent confirmation. The area of $\mathcal A_C$, computed by integrating over $u$ the total length of the section described in Section~4, is $32.329$ for $u\in[-20,20]$ with $6000$ values of $u$, and $32.349$ for $u\in[-40,40]$ with $8000$ values, so about $32.35$ up to an error that we estimate to be below $0.05$. The maximal value for this polygon is $\pi^2\Area(\Delta)=4\pi^2\approx39.478$ by \cite{PR04}, so $\mathcal A_C$ fills about $82\%$ of it. By \cite{MR2001} a curve with this Newton polygon and this amoeba cannot be Harnack, and since both areas are preserved by toric equivalence it cannot be toric-equivalent to a Harnack curve either.


\end{document}